\documentclass[11pt]{amsart}
\usepackage{amsfonts,amssymb,mathrsfs,latexsym, amsthm}
\usepackage{amsmath,amscd}
\usepackage[all]{xy}
\usepackage{color}

\newtheorem{thm}{Theorem}
 \newtheorem{cor}{Corollary}
 \newtheorem{lem}{Lemma}
 \newtheorem{prop}{Proposition}

 \theoremstyle{definition}
 \newtheorem{defn}{Definition}
 \newtheorem{ex}{Example}
 \newtheorem{rem}{Remark}
 
\newcommand{\Hom}{\operatorname{Hom}}

\newcommand{\wh}{\widehat}
\newcommand{\mbb}{\mathbb}
\newcommand{\mbf}{\mathbf}
\newcommand{\mcal}{\mathcal}
\newcommand{\mrk}{\mathfrak}

\newcommand{\ff}{{}'\!\mrk f}

\title[Deformation and Drinfeld double]{Drinfeld double of deformed quantum algebras}
\makeindex

\author[Z. Fan and J. Xing]{Zhaobing Fan and Junjing Xing}

\address{College of Automation and College of Science, Harbin Engineering University, Harbin 150001, China}
\email{fanz@math.ksu.edu  (Fan)}
\address{College of Automation, Harbin Engineering University, Harbin 150001, China}
\email{xingjunjing2017@hrbeu.edu.cn (Xing)}

\date{\today}

\begin{document}
\begin{abstract}
 We provide a deformation, $\mathfrak{f}_{\beta}$, of Lusztig algebra $\mathbf{f}$. Various quantum algebras in literatures, 
 including  half parts of two-parameter quantum algebras, quantum superalgebras, and multi-parameter quantum algebras/superalgebras, 
 are all specializations of $\mathfrak{f}_{\beta}$. Moreover, 
 $\mathfrak{f}_{\beta}$ is isomorphic to Lusztig algebra $\mathbf{f}$ up to a twist. 
 As a consequence, half parts of those quantum algebras are isomorphic to Lusztig algebra $\mathbf{f}$ 
 over a big enough ground field up to certain twists. 
 We further construct the entire algebra $\mbf U_{\beta,\xi}$ by Drinfeld double construction. As special cases, 
 above quantum algebras all admit a Drinfeld double construction under certain assumptions.
\end{abstract}
\maketitle

\setcounter{tocdepth}{1}
\tableofcontents

\section{Introduction}
In \cite{FL1}, the first author and Li constructed an algebra by using mixed perverse sheaves on representation varieties of a quiver,
and showed  that the algebra is isomorphic to the negative part of two-parameter quantum algebras defined in literatures.
As a byproduct, the half part of two-parameter quantum algebras is isomorphic to that of Lusztig algebra $\mathbf{f}$ 
defined in \cite[Chapter 1]{Lusztigbook} after a twist involving only the second parameter.
Later on, in \cite{CFLW}, Clark-Fan-Li-Wang showed that under certain assumptions, 
the half part of quantum superalgebras is isomorphic to its non-super analogue after a twist.
These results provided a novel way to study problems related to two-parameter quantum algebras/quantum superalgebras,
for example, in \cite{C}, the author showed that knot invariants from quantum algebra and its super analogy are the same using the twist in \cite{CFLW}.

It's a natural question how far the above results can be generalized to? In this paper, 
we provide a deformation $\mathfrak{f}_{\beta}$ of the Lusztig algebra $\mathbf{f}$ 
associated with a Cartan datum $(I,\cdot)$ over $\mathbb{F}$, where $\mathbb{F}$ is a ground field containing $\mbb Q(v)$ and $\alpha,\beta:\mathbb{N}[I]\times\mathbb{N}[I]\rightarrow\mathbb{F}$ are two multiplicative bilinear forms. 
Various quantum algebras in literatures, including two-parameter quantum algebras in \cite{FL1}, 
quantum superalgebras in \cite{CHW1}, multi-parameter quantum algebras in \cite{HPR} and multi-parameter quantum superalgebras in \cite{KKO13}, 
are all specializations of $\mathfrak{f}_{\beta}$. We further show the following theorem (Theorem \ref{thm1} in Section \ref{sec2.2}).
\\

\noindent$\textbf{Theorem\ A}$\  $For\ all\ homogenous\ elements\ x,y\in\mathfrak{f}_{\beta},\ let\ x*y=\gamma(|x|,|y|)$
$xy$,
$\ then$
$$(\mathfrak{f}_{\beta}\ ,\ *)\simeq \mathbf{f}\otimes_{\mathbb{Q}(v)}\mathbb{F},$$
$where\ \gamma:\mathbb{N}[I]\times\mathbb{N}[I]\rightarrow\mathbb{F}\ is\ a\ multiplicative\ bilinear\ form\ satisfying\ \eqref{gamma}.$
\\

As special cases, Theorem A shows that half parts of those quantum algebras are isomorphic to Lusztig algebra $\mathbf{f}$ after certain twists.
This partially answers the earlier question we raised.
Theorem A also indicates that there might be some relationships among the representations of those quantum algebras.
 In fact, in \cite{FLL}, Fan-Li-Lin proved that the categories of weight modules of those quantum algebras are all equivalent.

We note that the twisted $*$-product on $\mathfrak{f}_{\beta}$ is a variant of the deformations defined in \cite[Section 4]{FL1} 
(resp. \cite[Section 2]{CFLW}) relating one-parameter quantum groups to its two-parameter analogy (resp. super analogy). 
Surprisingly, the choice of $\gamma(-,-)$ in Theorem A is not unique. So Theorem A not only generalizes the results in \cite{FL1} 
(resp. \cite{CFLW}) to other quantum algebras, but also discovers new phenomena for two-parameter quantum algebras (resp. quantum superalgebras).

Drinfeld \cite{Dri87} built a braided Hopf algebra out of any finite-dimensional Hopf algebra with the invertible antipode, 
which is called ``the Drinfeld double construction'', providing the solution of the Yang-Baxter equation. 
In recent years, Drinfeld doubles have been studied by various authors 
as a useful tool of recovering the entire quantum groups from half parts.  
Xiao \cite{Xiao1} got a complete realization of quantum groups 
by applying the reduced form of the Drinfeld double of a Hall algebra to the generic composition algebra.
 Benkart-Witherspoon \cite{BW} characterized two-parameter quantum groups of type A 
 as the Drinfeld double of the certain Hopf subalgebras with respect to a skew-dual paring. 
 Following this work, Bergeron-Gao-Hu \cite{BGH} provided the Drinfeld double construction of 
 two-parameter quantum groups of types B, C and D.
 Moreover,  Hu-Rosso-Zhang \cite{HRZ}, Hu-Zhang \cite{HZ} and Gao-Hu-Zhang \cite{GHZ}  
 gave the Drinfeld double construction of two-parameter quantum affine algebras 
 $U_{r,s}(\widehat{\mathfrak{sl}_n})$, $U_{r,s}(C_n^{(1)})$ and $U_{r,s}(G_2^{(1)})$, respectively.

In this paper, we study the entire quantum group of the algebra $\mathfrak{f}_{\beta}$ 
by constructing a  Drinfeld double $\mbf U_{\beta,\xi}$, 
where $\xi:\mathbb{N}[I]\times\mathbb{N}[I]\rightarrow\mathbb{F}$ is a symmetric multiplicative bilinear form. 
As special cases, two-parameter quantum algebras, quantum superalgebras, 
and multi-parameter quantum algebras/superalgebras all admit a Drinfeld double construction.

This paper is organized as follows. In Section \ref{sec2}, we provide the deformation $\mathfrak{f}_{\beta}$ of 
Lusztig algebra $\mathbf{f}$ in ~\cite[Chapter 1]{Lusztigbook} 
and show that $\mathfrak{f}_{\beta}$ is isomorphic to 
the ordinary quantum algebra up to a twist. In Section \ref{sec3}, 
we construct the Drinfeld double, $\mbf U_{\beta,\xi}$, 
of the algebra  $\mathfrak{f}_{\beta}$ as a Hopf algebra. In Section \ref{sec4}, 
we provide various specializations of $\mbf U_{\beta,\xi}$ 
which are shown to be isomorphic to various quantum algebras in literatures.\vspace{8pt}

\noindent\textbf{Acknowledgements.}The first author thanks Zongzhu Lin for earlier discussions which motivate this project.
We thank Yiqiang Li, Sean Clark and Weiqiang Wang for fruitful collabrations, which provided some examples for this project.
Z. Fan is partially supported by the NSF of China grant 11671108, the NSF of Heilongjiang
Province grant LC2017001 and the Fundamental Research Funds for the central universities GK2110260131.

\section{Deformation of half part of quantum algebras}\label{sec2}
In this section, we provide a deformation, $\mathfrak{f}_{\beta}$, of the Lusztig algebra in ~\cite[Chapter 1]{Lusztigbook} 
depending on two multiplicative bilinear forms $\alpha,\beta:\mathbb{N}[I]\times\mathbb{N}[I]\rightarrow\mathbb{F}$ and 
study some examples of $\mathfrak{f}_{\beta}$.

\subsection{Deformation of the Lusztig algebra}\label{sec2.1}
Let us fix a Cartan datum $(I,\cdot)$.
We simply write $d_i=\frac{i\cdot i}{2}$ and $a_{ij}=2\frac{i\cdot j}{i\cdot i}$.
Let $v$ be an indeterminate. We set $v_i=v^{d_i}$ for all $i\in I$. 
Given a field $\mbb F$ containing $\mbb Q(v)$, let $\alpha, \beta: \mbb N[I]\times \mbb N[I] \rightarrow \mbb F$
be two multiplicative bilinear forms satisfying the following properties:
\begin{itemize}
  \item[(i)] $\beta(i,j)\alpha(j,i)=\beta(j,i)\alpha(i,j),\quad \forall i,j\in I$;
  \item[(ii)] $\beta(i,j)\beta(j,i)=1,\quad \forall i,j\in I$, we further assume that $\beta(i,i)=1$.
\end{itemize}

Let ${}'\!\mrk f$ be the free unital associative $\mbb F$-algebra generated by $\{\theta_i, i\in I\}$.
The algebra ${}'\!\mrk f$ is an $\mbb N[I]$-graded algebra if we set the grading of $\theta_i$ to be $i$.
For any homogenous element $x\in {}'\!\mrk f$, we denote by $|x|$ the grading of $x$.

We define an algebra structure on ${}'\!\mrk f\otimes {}'\!\mrk f$ by
\[
(x_1\otimes x_2)(y_1\otimes y_2)=v^{-|y_1|\cdot |x_2|}\beta(|x_2|,|y_1|) x_1y_1\otimes x_2y_2,
\]
for any homogenous elements $x_2, y_1\in {}'\!\mrk f$.

Let $r:\ff\rightarrow \ff\otimes \ff$ be the $\mbb F$-algebra homomorphism such that
\[r(\theta_i)=\theta_i\otimes 1+1\otimes \theta_i, \quad \forall i\in I.\]

\begin{prop}
  There is a unique symmetric bilinear form $(-,-): \ff\times \ff\rightarrow \mbb F$ such that
  \begin{itemize}
   \item[(a)] $(1,1)=1$ and $(\theta_i, \theta_j)=\delta_{ij}(1-v_i^{-2})^{-1},\quad \forall i,j\in I$;
   \item[(b)] $(x, y'y'')=(r(x), y'\otimes y''),\quad \forall x, y', y''\in \ff$;
   \item[(c)] $(x'x'', y)=(x'\otimes x'', r(y)),\quad \forall x', x'', y\in \ff$.
  \end{itemize}
  Here the bilinear form on $\ff\otimes \ff$ is defined by
  \[(x_1\otimes x_2, y_1\otimes y_2)=\alpha(|x_1|, |x_2|)(x_1, y_1)(x_2, y_2),\]
  for any homogenous elements $x_1, x_2, y_1, y_2 \in \ff$.
\end{prop}

\begin{proof}
  The proof  for Proposition 13 in \cite{FL1} works through if we
  replace $t^{2(2\delta_{ij}\Omega_{ii}-\Omega_{ij})}$ and $t^{\Omega_{ji}-\Omega_{ij}}$ by corresponding bilinear forms $\alpha(i,j)$ and $\beta(i,j)$,
  respectively. For part (c), one need use the property (i).
\end{proof}

Let $\mrk J$ be the radical of the bilinear form $(-,-)$.
It is clear that $\mrk J$ is a two-sided ideal of $\ff$.
Denote the quotient algebra of $\ff$ by $\mrk J$ by \[\mrk f_{\beta}=\ff/\mrk J.\]

For any $i\in I$, let $r_i ({\rm resp.}\ {}_ir): \ff\rightarrow \ff$
be the unique linear map satisfying the following properties:
\begin{equation*}
  \begin{split}
    r_i(1)=0,\ r_i(\theta_j)=\delta_{ij},\ \forall j\in I\ {\rm and}\
    r_i(xy)=v^{-i\cdot |y|}\beta(i, |y|)r_i(x)y+xr_i(y);\\
    {}_ir(1)=0,\ {}_ir(\theta_j)=\delta_{ij},\ \forall j\in I\ {\rm and}\
    {}_ir(xy)={}_ir(x)y+v^{-i\cdot |x|}\beta(|x|,i)x{}_ir(y).
  \end{split}
\end{equation*}
By an induction on $|x|$, we can show that $r(x)=r_i(x)\otimes \theta_i$ (resp. $r(x)= \theta_i\otimes{}_ir(x)$)
plus other terms.
By this property, we have the following lemma.
\begin{lem}
  $(a)$ For any $x, y\in \ff$, we have
  \[(y\theta_i, x)=\alpha(|y|, i)(y, r_i(x))(\theta_i, \theta_i)\ {\rm and}\
  (\theta_iy, x)=\alpha( i, |y|)(\theta_i, \theta_i)(y, {}_ir(x)).\]
  $(b)$ For any $i\in I$, the linear maps ${}_ir, r_i: \ff\rightarrow \ff$ send $\mrk J$ to itself.
\end{lem}
By using the same argument as that for Lemma 1.2.15 in \cite{Lusztigbook}, we have the following lemma.
\begin{lem}\label{lemrix}
  If a homogenous element $x\in \mrk f_{\beta}$ satisfies that $r_i(x)=0$ (resp. ${}_ir(x)=0$) for all $i\in I$ , then $x=0$.
\end{lem}
For any $n\in \mbb N$ and $c\in \mbb F$, we set
\[[n]_c=\frac{c^n-c^{-n}}{c-c^{-1}},\quad [n]_c^!=\prod_{k=1}^n[k]_c \quad {\rm and}\quad
\begin{bmatrix}
  n\\k
\end{bmatrix}_c
=\frac{[n]_c^!}{[k]_c^![n-k]_c^!}.\]
Let
\[\theta_{i}^{(n)}=\frac{\theta_i^n}{[n]_{v_i}^!}.\]
By using $(1\otimes \theta_i)(\theta_i\otimes 1)=v_i^{-2}(\theta_i\otimes 1)(1\otimes \theta_i)$ 
and the quantum binomial formula in \cite[Section 1.3.5]{Lusztigbook},
 we have the following lemma.
\begin{lem}\label{lemirtheta}
 We have ${}_ir(\theta_{i}^{(n)})=v_i^{-(n-1)}\theta_{i}^{(n-1)}$ for any $n\in \mbb N$ and  $i\in I$.
\end{lem}
\begin{prop}\label{propserrerelation}
The generators $\theta_i$ of $\mrk f_{\beta}$ satisfy the following relations.
\[\sum_{k+k'=1-a_{ij}}(-1)^k\beta(i,j)^{-k}\theta_i^{(k)}\theta_j\theta_i^{(k')}=0, \quad \forall i\neq j\in I.\]
\end{prop}

\begin{proof}
For simplicity, we write
\begin{equation}\label{Dij}
  D_{ij}:=\sum_{k+k'=1-a_{ij}}(-1)^k\beta(i,j)^{-k}\theta_i^{(k)}\theta_j\theta_i^{(k')}, \quad \forall i\neq j\in I.
  \end{equation}
  By Lemma \ref{lemrix}, we only need to show that ${}_lr(D_{ij})=0$ for any $l\in I$. It is clear that
  \begin{equation}\label{lrDij}
  {}_lr(D_{ij})=0,\quad {\rm if\ } l\not =i, j.
  \end{equation}
  By Lemma \ref{lemirtheta} and the definition of ${}_ir$, we have
  \begin{eqnarray*}
  & &{}_ir(\theta_i^{(k)}\theta_j\theta_i^{(k')})
  ={}_ir(\theta_i^{(k)}\theta_j)\theta_i^{(k')}+v^{-i\cdot (ki+j)}\beta(ki+j,i)v_{i}^{-(k'-1)}\theta_i^{(k)}\theta_j\theta_i^{(k'-1)}\\
  &=&v_{i}^{-(k-1)}\theta_i^{(k-1)}\theta_j\theta_i^{(k')}+v^{-i\cdot (ki+j)}\beta(j,i)v_{i}^{-(k'-1)}\theta_i^{(k)}\theta_j\theta_i^{(k'-1)}\\
  &=&v_{i}^{-(k-1)}\theta_i^{(k-1)}\theta_j\theta_i^{(k')}+v_i^{-k}\beta(j,i)\theta_i^{(k)}\theta_j\theta_i^{(k'-1)}.
  \end{eqnarray*}
  So  ${}_ir(D_{ij})$ is equal to
\begin{eqnarray*}
   \sum_{1\leq k\leq 1-a_{ij}}(-1)^k v_{i}^{-(k-1)}\beta(i,j)^{-k}\theta_i^{(k-1)}\theta_j\theta_i^{(k')}\\
   +\sum_{0\leq k\leq -a_{ij}}(-1)^k v_{i}^{-k}\beta(i,j)^{-(k+1)}\theta_i^{(k)}\theta_j\theta_i^{(k'-1)}\\
   =\sum_{0\leq k\leq -a_{ij}}(-1)^{(k+1)} v_{i}^{-k}\beta(i,j)^{-(k+1)}\theta_i^{(k)}\theta_j\theta_i^{(k'-1)}\\
   +\sum_{0\leq k\leq -a_{ij}}(-1)^k v_{i}^{-k}\beta(i,j)^{-(k+1)}\theta_i^{(k)}\theta_j\theta_i^{(k'-1)}. \nonumber
\end{eqnarray*}
By comparing the exponents of $v_i$ and $\beta(i,j)$, we have
\begin{equation}\label{irDij}
  {}_ir(D_{ij})=0.
\end{equation}
  By Lemma \ref{lemirtheta} and the definition of ${}_jr$ again, we have
   \begin{equation*}
   {}_jr(\theta_i^{(k)}\theta_j\theta_i^{(k')})
   = v^{-kj\cdot i}\beta(ki,j)\theta_i^{(k)}\theta_i^{(k')}.
  \end{equation*}
  So ${}_jr(D_{ij})$ is equal to
  \begin{eqnarray*}
  & &\sum_{k+k'=1-a_{ij}}(-1)^k v^{-kj\cdot i}\theta_i^{(k)}\theta_i^{(k')}\\
  &=&\sum_{k+k'=1-a_{ij}}(-1)^kv_i^{-ka_{ij}}\begin{bmatrix}
     1-a_{ij}\\k
   \end{bmatrix}_{v_i}\theta_i^{(1-a_{ij})}.
  \end{eqnarray*}
By Section 1.3.4 in~\cite{Lusztigbook},
  \begin{eqnarray*}
  & &\sum_{k+k'=1-a_{ij}}(-1)^kv_i^{-ka_{ij}}\begin{bmatrix}
    1-a_{ij}\\k
  \end{bmatrix}_{v_i}=0.
  \end{eqnarray*}
  So we have
  \begin{equation}\label{jrDij}
  {}_jr(D_{ij})=0.
\end{equation}

Proposition follows $\eqref{lrDij}$,$\eqref{irDij}$ and $\eqref{jrDij}$.
\end{proof}
\subsection{Relation to the Lusztig algebra}\label{sec2.2}
Recall the definition of $\beta(-,-)$ from the Section \ref{sec2.1}. We assume that there exits a multiplicative bilinear form 
$\gamma: \mbb N[I]\times \mbb N[I] \rightarrow \mbb F$ satisfying the following property:
\begin{equation}\label{gamma}
 \gamma(i,j)=\gamma(j,i)\beta(j,i),\quad \forall i, j\in I.
\end{equation}

We define a new multiplication $``*"$ on $\mrk f_{\beta}$ as follows.
\[x*y=\gamma(|x|,|y|) xy, \quad \text{for all homogenous elements}\ x, y\in \mrk f_{\beta}.\]

Recall that $\mathbf{f}$ is Lusztig algebra defined in ~\cite[Chapter 1]{Lusztigbook} associated to a Cartan datum $(I, \cdot)$.
Let $\mathbf{f}_{\mbb F}= \mathbf{f}\otimes_{\mbb Q(v)}\mbb F$.
\begin{thm}\label{thm1}
 If there exists $\gamma(-,-)$ satisfying \eqref{gamma},
then the map $\eta:(\mrk f_{\beta}, *)\rightarrow \mathbf{f}_{\mbb F}$ sending $\theta_i$ to $\theta_i$ for all $i\in I$ is an algebra isomorphism.
\end{thm}

\begin{proof}
By using \eqref{gamma} and Proposition \ref{propserrerelation},
the algebra $(\mrk f_{\beta}, *)$ is generated by $\theta_i, i\in I$ and subjects to the  following relations:
 \begin{equation*}
  D'_{ij}:=\sum_{k+k'=1-a_{ij}}(-1)^k\begin{bmatrix}
     1-a_{ij}\\k
   \end{bmatrix}_{v_i}\theta_i^{*k}*\theta_j*\theta_i^{*k'}=0.
  \end{equation*}

Thus,
\begin{equation*}
  \eta(D'_{ij})=\sum_{k+k'=1-a_{ij}}(-1)^k\begin{bmatrix}
     1-a_{ij}\\k
   \end{bmatrix}_{v_i}\theta_i^{k}\theta_j\theta_i^{k'}=0.
  \end{equation*}
This finishes the proof.
\end{proof}
\begin{rem}\label{rem}
   In general, the choice of $\gamma(-,-)$ satisfying \eqref{gamma} is not unique if it exists.
    Therefore the twist in Theorem \ref{thm1} is not unique.
 \end{rem}

\subsection{Examples}\label{sec2.3}
In this subsection, we shall give some examples of $\mrk f_{\beta}$.
Although the results for Examples \ref{extwoparameter} and \ref{exsuper} have been found in \cite{FL1} and \cite{CFLW}, respectively,
by Remark \ref{rem}, the twists in  Examples \ref{extwoparameter} and \ref{exsuper} are not unique.
We shall provide a new twist for two-parameter quantum algebras (resp. quantum superalgebras).
To the best of our knowledge, results for Examples \ref{exmul}, \ref{exmulsuper1} and \ref{exmulsuper2} are new in literatures.
\begin{ex}\label{extwoparameter}
We briefly review the two-parameter quantum algebra $\mathfrak f_{v,t}$ defined in \cite{FL1}, where $t$ is the second indeterminate.

For a fixed Cartan datum $(I, \cdot)$, let $\Omega=(\Omega_{ij})_{i,j\in I}$ be an integer matrix satisfying that
\begin{itemize}
  \item[(a)] $\Omega_{ii} \in \mathbb{Z}_{>0}$, $\Omega_{ij}\in \mathbb{Z}_{\leq 0}$ for all $i\neq j \in I$;
  \item[(b)] $\frac{\Omega_{ij}+\Omega_{ji}}{\Omega_{ii}}\in \mathbb{Z}_{\leq 0}$ for all $i\neq j \in I$;
  \item[(c)] the greatest common divisor of all $\Omega_{ii}$ is equal to 1;
  \item[(d)] $i\cdot j=\Omega_{ij} +\Omega_{ji},  \quad \forall i, j\in I.$
\end{itemize}

The two-parameter quantum algebra
$\mathfrak f_{v,t}$ associated to $\Omega$ is the associative
$\mbb Q(v,t)$-algebra generated by $\theta_i, \forall i\in I$ and subjects to the following relations.
\begin{equation*}\label{eqserre}
      \sum_{k+k'=1-a_{ij}}(-1)^k\begin{bmatrix}1-a_{ij}\\k\end{bmatrix}_{v_i}
      t^{k(\Omega_{ij}-\Omega_{ji})}\theta_i^{k}\theta_j\theta_i^{k'}=0, \quad \forall i\neq j\in I.
    \end{equation*}

We set $\mathbb{F}=\mathbb{Q}(v,t)$ and
\begin{equation}\label{eqtwoparameter}
\beta(i,j)=t^{\Omega_{ji}-\Omega_{ij}},\quad\alpha(i,j)=t^{2(2\delta_{ij}\Omega_{ii}-\Omega_{ij})},\quad\forall i,j\in I.
\end{equation}
We further set $\gamma(i,j)=t^{\Omega_{ij}-2\delta_{ij}\Omega_{ii}}, \ \forall i,j\in I$. Under these settings, 
Properties (i), (ii) in Section \ref{sec2.1} and \eqref{gamma} hold, and the construction in Section \ref{sec2.1} 
turns into the construction in ~\cite[Section 3]{FL1}. Under the specialization \eqref{eqtwoparameter}, 
$\mrk f_{\beta}$ is exactly the algebra $\mathfrak f_{v,t}$. By Theorem \ref{thm1}, $(\mathfrak f_{v,t}, \ast)$ is isomorphic to 
$\mathbf{f} \otimes_{\mbb Q(v)}\mbb Q(v,t)$, where $x \ast y=\gamma(|x|,|y|)xy$.
This result has been found in \cite{FL1}. We note if we set
$$\gamma(i,j)=v^{i\cdot j}t^{\Omega_{ij}},\quad \forall i,j \in I,$$
the result in Theorem \ref{thm1} still holds.
\end{ex}

\begin{ex}\label{exsuper}
Recall the definition of the quantum superalgebra $\mathfrak{f}_{v,\mbf i}$ from \cite{CHW1}. 
We assume that $(I, \cdot)$ is a bar-consistent Cartan super datum. By bar-consistent, we mean that $d_i$ is odd if and only if $i$ is odd.
Let $\mathcal{P}: I\rightarrow \mbb Z_2$ be the parity function such that
$\mathcal{P}(i)=1$ if $i$ is odd and 0 otherwise.
We set $t=\mbf i$, the complex number such that $\mbf i^2=-1$, and $q=v^{-1}t$.

The quantum superalgebra $\mathfrak{f}_{v,\mbf i}$ associated to $(I,\cdot)$ is the associative $\mbb Q(v, \mathbf{i})$-algebra
generated by $\theta_i, \forall i\in I$ and subjects to the following relations.
\begin{equation*}\label{eqserresuper}
  \sum_{k+k'=1-a_{ij}}(-1)^k\begin{bmatrix}1-a_{ij}\\k\end{bmatrix}_{v_i}
  t^{k i\cdot j+2k\mathcal{P}(i)\mathcal{P}(j)}\theta_i^{k}\theta_j\theta_i^{k'}=0,\quad \forall i\neq j\in I.
\end{equation*}

We set $\mathbb{F}=\mathbb{Q}(v,\mathbf{i})$ and
$$\beta(i,j)=t^{i\cdot j}t^{2\mathcal{P}(i)\mathcal{P}(j)},\quad \alpha(i,j)=1,\quad \forall i, j\in I.$$
In particular, $\beta(i,i)=1$. We further fix an order $``<"$ on $I$ and set
 \[ \gamma(i, j)=\left\{\begin{array}{ll}
   t^{i\cdot j}& {\rm if}\ j<i,\\
   t^{d_i}&{\rm if}\ j=i,\\
   t^{2\mathcal{P}(i)\mathcal{P}(j)}& {\rm if}\ j>i.
 \end{array}
 \right.
 \]
 Under these settings, Properties (i), (ii) in Section \ref{sec2.1} and \eqref{gamma} hold, 
 and the construction in Section \ref{sec2.1} turns into the construction in \cite{CHW1}.
Moreover, we have
$$\mrk f_{\beta}=\mathfrak{f}_{v,\mbf i}.$$
 By Theorem \ref{thm1}, $(\mathfrak{f}_{v,\mbf i}, \ast)$ is isomorphic to $\mathbf{f} \otimes_{\mbb Q(v)}\mbb Q(v,\mathbf{i})$, 
 where $x \ast y=\gamma(|x|,|y|)xy$. This result has been found in \cite{CFLW}. If we set
\[ \gamma(i, j)=\left\{\begin{array}{ll}
   t^{2i\cdot j}& {\rm if}\ j<i,\\
   t^{d_i}&{\rm if}\ j=i,\\
   t^{i\cdot j}t^{2\mathcal{P}(i)\mathcal{P}(j)}& {\rm if}\ j>i,
 \end{array}
 \right.
 \]
 the result in Theorem \ref{thm1} still holds.
\end{ex}

\begin{ex}\label{exmul}
Recall the definition of the multi-parameter quantum algebra $\mrk f_{v,\mathbf{q}}$ from \cite{HPR}.
For a fixed Cartan datum $(I, \cdot)$, let $A=(a_{ij})_{i,j\in I}$ be an associated generalized Cartan matrix and $q_{ij},
\ \forall i,j\in I,$ be indeterminates over $\mathbb{Q}$ such that
$q_{ij}q_{ji}=q_{ii}^{a_{ij}}.$ We set $\mathbf{q}=(q_{ij})_{i,j\in I}$ and further assume that
\begin{equation}\label{qii}
q_{ii}=v_i^{-2},\quad \forall i \in I.
\end{equation}

The multi-parameter quantum algebra $\mrk f_{v,\mathbf{q}}$ associated to $(I, \cdot)$ is the associative  
$\mathbb{Q}(v,q_{ij}^{1/2})$-algebra generated by $\theta_i, \forall i\in I$ and subjects to the following relations.
\begin{equation*}\label{eqserremulti}
  \sum_{k+k'=1-a_{ij}}(-1)^k\begin{bmatrix}1-a_{ij}\\k\end{bmatrix}_{v_i}
  v^{-k i\cdot j}q_{ij}^{-k}\theta_i^{k}\theta_j\theta_i^{k'}=0,\quad \forall i\neq j\in I.
\end{equation*}

We set $\mathbb{F}=\mathbb{Q}(v,q_{ij}^{1/2})$ and
\begin{equation}\label{eqmul}
\beta(i,j)=v^{i\cdot j}q_{ij},\quad\alpha(i,j)=q_{ij},\quad\gamma(i,j)=q_{ii}^{\delta_{ij}}q_{ji}^{1/2},\quad\forall i, j\in I.
\end{equation}
In particular, $\beta(i,i)=1$. Under the settings in \eqref{eqmul},$$\mrk f_{\beta}=\mrk f_{v,\mathbf{q}}.$$
By Theorem \ref{thm1}, $(\mrk f_{v,\mathbf{q}}, \ast)$ is isomorphic to $\mathbf{f} \otimes_{\mbb Q(v)}\mathbb{Q}(v,q_{ij}^{1/2})$ 
under the assumption \eqref{qii}, where $x \ast y=\gamma(|x|,|y|)xy$.
\end{ex}

\begin{ex}\label{exmulsuper1}
Recall the definition of the multi-parameter quantum superalgebra $\mrk f_{\mathbf{s},\mathbf{p}}$ from \cite[Section 2]{KKO13}. 
For a fixed  bar-consistent super Cartan datum $(I, \cdot)$, let $\mathbf{s}:=\{s_{ij}\}_{i,j\in I}$ and 
$\mathbf{p}:=(\{p_{ij}\}_{i,j\in I}, \{p_i\}_{i\in I})$ be families of invertible elements of a commutative
ring $\mathbb{Q}(p_i^{1/(2d_i)},s_{ij},p_{ij})$ such that
\begin{equation}\label{con}
p_{ij}^2=p_i^{2a_{ij}},\quad (p_{ij}p_{ji})/(s_{ij}s_{ji})=p_i^{2a_{ij}},\quad {\rm and}\quad p_{ii}/s_{ii}=p_i^2,\ \forall i, j\in I.
\end{equation}

The condition \eqref{con} implies
\begin{equation}\label{conapply}
p_i^{2a_{ij}}=p_j^{2a_{ji}},\quad (s_{ij}s_{ji})^2=1,\quad {\rm and}\quad s_{ij}s_{ji}=p_{ij}^{-1}p_{ji}.
\end{equation}

We further assume that
\begin{equation}\label{pi}
p_i=v_ih_i\quad {\rm such\  that}\quad h_i^2=1,\quad \forall i\in I.
\end{equation}
The multi-parameter quantum superalgebra $\mrk f_{\mathbf{s},\mathbf{p}}$
associated to $(I,\cdot)$ is the associative $\mbb Q(v^{1/2},s_{ij},p_{ij})$-algebra generated by $\theta_i, \forall i\in I$ 
and subjects to the following relations.
\begin{equation*}\label{eqserremultisuperI}
  \sum_{k+k'=1-a_{ij}}(-1)^k\begin{bmatrix}1-a_{ij}\\k\end{bmatrix}_{v_i}
  s_{ji}^{k}p_{ji}^{-k}v_i^{ka_{ij}}\theta_i^{k}\theta_j\theta_i^{k'}=0,\quad \forall i\neq j\in I.
\end{equation*}

We set $\mathbb{F}=\mathbb{Q}(v^{1/2},s_{ij},p_{ij})$ and
\begin{equation}\label{eqmulsuper1}
\beta(i,j)=s_{ji}^{-1}p_{ji}v_i^{-a_{ij}},\quad\alpha(i,j)=s_{ij}p_{ij}^{-1},\quad\forall i, j\in I.
\end{equation}

In particular, $\beta(i,i)=1$. We further fix an order $``<"$ on $I$ and set
 \[ \gamma(i, j)=\left\{\begin{array}{ll}
   s_{ij}^{-1}h_{i}^{a_{ij}}s_{ji}& {\rm if}\ j<i,\\
   s_{ii}&{\rm if}\ j=i,\\
   p_{ij}p_j^{-a_{ji}}s_{ji}^{-1}& {\rm if}\ j>i.
 \end{array}
 \right.
 \]
 Under the settings in \eqref{eqmulsuper1}, $$\mrk f_{\beta}=\mrk f_{\mathbf{s},\mathbf{p}}.$$
 By Theorem \ref{thm1}, $(\mrk f_{\mathbf{s},\mathbf{p}}, \ast)$ is isomorphic to 
 $\mathbf{f} \otimes_{\mbb Q(v)}\mbb Q(v^{1/2},s_{ij},q_{ij})$ under the assumption \eqref{pi}, where $x \ast y=\gamma(|x|,|y|)xy$.
 \end{ex}

\begin{ex}\label{exmulsuper2}
Recall the definition of the multi-parameter quantum superalgebra 
$\mrk f_{\tilde{\mathbf{s}},\tilde{\mathbf{p}}}$ from \cite[Section 3]{KKO13}. 
For a fixed bar-consistent super Cartan datum $(I, \cdot)$, let $\tilde{\mathbf{s}}:=\{\tilde{s}_{ij}\}_{i,j\in I}$
 and $\tilde{\mathbf{p}}:=\{\tilde{p}_i\}_{i\in I}$ be families of invertible elements of a commutative ring 
 $\mathbb{Q}(\tilde{p}_i^{1/(2d_i)},\tilde{s}_{ij})$ such that
$$\tilde{s}_{ij}\tilde{s}_{ji}=\tilde{p}_i^{-a_{ij}},\quad \ {\rm and}\ \quad \tilde{s}_{ii}=\tilde{p}_i^{-1},\ \forall i, j\in I.$$
We further assume that
\begin{equation}\label{tildepi}
\tilde{p}_i=v_i^2,\quad \forall i\in I.
\end{equation}
The multi-parameter quantum superalgebra $\mrk f_{\tilde{\mathbf{s}},\tilde{\mathbf{p}}}$
associated to $(I,\cdot)$ is the associative $\mbb Q(v,\tilde{s}_{ij})$-algebra generated by $\theta_i, \forall i\in I$ 
and subjects to the following relations.
\begin{equation*}\label{eqserremultisuperII}
  \sum_{k+k'=1-a_{ij}}(-1)^k\begin{bmatrix}1-a_{ij}\\k\end{bmatrix}_{v_i}
  (\tilde{s}_{ij}\tilde{p}_i^{a_{ij}/2})^{-k}\theta_i^{k}\theta_j\theta_i^{k'}=0,\quad \forall i\neq j\in I.
\end{equation*}

We set $\mathbb{F}=\mathbb{Q}(v,\tilde{s}_{ij})$ and
\begin{equation}\label{eqmulsuper2}
\beta(i,j)=\tilde{s}_{ij}(\tilde{p}_i)^{a_{ij}/2},\quad  \alpha(i,j)=\tilde{s}_{ij},\quad\forall i, j\in I.
\end{equation}
In particular, $\beta(i,i)=1$. We further fix an order $``<"$ on $I$ and set
 \[ \gamma(i, j)=\left\{\begin{array}{ll}
   \tilde{s}_{ji}& {\rm if}\ j<i,\\
   \tilde{s}_{ii}^3&{\rm if}\ j=i,\\
   \tilde{p}_i^{-a_{ij}/2}& {\rm if}\ j>i.
 \end{array}
 \right.
 \]
 Under the settings in \eqref{eqmulsuper2}, $$\mrk f_{\beta}=\mrk f_{\tilde{\mathbf{s}},\tilde{\mathbf{p}}}.$$
 By Theorem \ref{thm1}, $(\mrk f_{\tilde{\mathbf{s}},\tilde{\mathbf{p}}}, \ast)$ is isomorphic to 
 $\mathbf{f} \otimes_{\mbb Q(v)}\mbb Q(v,\tilde{s}_{ij})$ under the assumption \eqref{tildepi}, 
 where $x \ast y=\gamma(|x|,|y|)xy$.

\end{ex}

\section{Drinfeld double of $\mrk f_{\beta}$}\label{sec3}

\subsection{}
We review briefly the Drinfeld double construction in ~\cite{Xiao1}.
For a Hopf algebra $A$, we shall denote by $\Delta_A, \varepsilon_A$
and $S_A$ its comultiplication, counit and antipode, respectively.

\begin{defn}
  \label{def1}~\cite[Section 2.2]{Xiao1}
  Given $\mbb F$-Hopf algebras $A$ and $B$, a skew-Hopf pairing of $A$ and $B$ is a bilinear form
  $\psi: A\times B\rightarrow \mbb F$ satisfying
   \begin{eqnarray*}\label{eq116}
\begin{split}
 (a)\quad & \psi(1, y)=\varepsilon_B(y),\quad \psi(x,1)=\varepsilon_A(x);\\
  (b)\quad & \psi(x, y'y'')=\psi(\Delta_A(x),y'\otimes y'');\\
  (c)\quad & \psi(x'x'', y)=\psi(x'\otimes x'',\Delta_B^{\rm op}(y));\\
  (d) \quad & \psi(S_A(x),y)=\psi(x, S_B^{-1}(y)),\quad \forall x, x',x''\in A,\; y,y',y''\in B,
\end{split}
\end{eqnarray*}
where $\psi(x'\otimes x'', y'\otimes y'')=\psi(x', y')\psi(x'',y'')$ and
 $\Delta_B^{\rm op}(y)=\sum y_2\otimes y_1$ if $\Delta_B(y)=\sum y_1\otimes y_2$.
\end{defn}
\begin{prop}
  \label{propdrinfelddouble}~\cite[Proposition 2.4]{Xiao1}
  Let $\psi: A\times B\rightarrow \mbb F$ be a skew-Hopf pairing. Then $A\otimes B$ has a Hopf algebra structure.
  Moreover, for any $x,x'\in A,\ y, y'\in B$, the algebra structure on $A\otimes B$ is defined by
  \begin{eqnarray*}
   \label{eq117}
\begin{split}
  {\rm (i)}\quad & (x\otimes 1)(x'\otimes 1)=xx'\otimes 1;\\
  {\rm (ii)}\quad & (1\otimes y)(1\otimes y')=1\otimes yy';\\
  {\rm (iii)}\quad & (x\otimes 1)(1\otimes y)=x\otimes y;\\
  {\rm (iv)} \quad & \textstyle (1\otimes y)(x\otimes 1)=\sum\psi(x_1, S_B(y_1))x_2\otimes y_2 \psi(x_3, y_3),
\end{split}
  \end{eqnarray*}
  where $\Delta_A^2(x)=\sum x_1\otimes x_2\otimes x_3$ and
  $\Delta_B^2(y)=\sum y_1\otimes y_2\otimes y_3$.

And the comultiplication $\Delta_{A\otimes B}$, counit $\varepsilon_{A\otimes B}$ and 
antipode $S_{A\otimes B}$ of $A\otimes B$ are given as follows.
\begin{equation*}
\begin{split}
&\Delta_{A\otimes B}(x\otimes y)=\textstyle \sum(x_1\otimes y_1)\otimes (x_2\otimes y_2),\\
&\varepsilon_{A\otimes B}(x\otimes y)=\varepsilon_A(x)\varepsilon_B(y),\\
&S_{A\otimes B}(x\otimes y)=(1\otimes S_B(y))(S_A(x)\otimes 1),\quad \forall x\in A,\; y\in B,
\end{split}
\end{equation*}
where $\Delta_A(x)=\sum x_1\otimes x_2$ and $\Delta_B(y)=\sum y_1\otimes y_2$.
\end{prop}

\subsection{The Hopf algebras  $\widehat{{}'\mathfrak{f}^+}$ and $\widehat{{}'\mathfrak{f}^-}$}

Let $\mathfrak{H}$ (resp. $\mathfrak{H}'$) be the commutative $\mbb F$-algebra generated by 
$ K_i^{\pm 1}$, $J_i^{\pm 1}$ (resp. $ K'^{\pm 1}_i, J'^{\pm 1}_i$) for all $i\in I$
and subject to the relations
$K_iK_i^{-1}=1$, $J_iJ_i^{-1}=1$ (resp. $K'_iK'^{-1}_i=1$, $J'_iJ_i'^{-1}=1$).

Let $\xi: \mbb N[I]\times \mbb N[I] \rightarrow \mbb F$
be a symmetric multiplicative bilinear form such that $\xi(i,i)^{1/2}\in \mathbb{F},\ \forall i \in I.$

Recall that $\beta(-,-)$ is a multiplicative bilinear form defined in Section \ref{sec2.1}. We set
$$\langle i,j\rangle=v^{-i\cdot j}\beta(i,j)\xi(j,i),\quad \forall i ,j\in I,$$
which is also a multiplicative bilinear form on $\mbb N[I]\times \mbb N[I]$.

To distinguish, we shall write $\ff^+$ (resp. $\ff^-$) the free $\mbb F$-algebra generated by $E_i$ (resp. $F_i$), for all $i\in I$.
We denote the algebra isomorphism by $\iota^{+}: {}'\!\mathfrak{f}\rightarrow {}'\!\mathfrak{f}^+$ 
with $ \iota^+(\theta_i)=E_i$ and the algebra anti-isomorphism by 
$\iota^{-}: {}'\!\mathfrak{f}\rightarrow {}'\!\mathfrak{f}^-$ with $ \iota^-(\theta_i)=F_i$.

We set $\widehat{{}'\mathfrak{f}^+}={}'\!\mathfrak{f}^+ \rtimes \mathfrak{H}$ and 
$\widehat{{}'\mathfrak{f}^-}={}'\!\mathfrak{f}^- \rtimes \mathfrak{H}'$, 
where the $\mathfrak{H}$-action (resp. $\mathfrak{H}'$-action) on ${}'\!\mathfrak{f}^+$ (resp. ${}'\!\mathfrak{f}^-$) is defined by
\begin{equation*}\label{eq102}
\begin{split}
K_i\cdot E_j:=K_iE_jK_i^{-1}=\langle i, j\rangle E_j,
\quad J_i&\cdot E_j:=J_iE_jJ_i^{-1}=\xi(j,i)E_j,\\
K'_i\cdot F_j:=K_i'F_jK_i'^{-1}=\langle j,i\rangle\xi(i,j)^{-1} F_j,
\ J'_i&\cdot F_j:=J'_iF_jJ_i'^{-1}=F_j, \forall i, j\in I.
\end{split}
\end{equation*}

 Moreover, the graded structure on $'\mathfrak{f}^+$ (resp. $'\mathfrak{f}^-$) can be extended to 
 $\widehat{'\mathfrak{f}^+}$ (resp. $\widehat{'\mathfrak{f}^-}$) by setting $|K_i|=|K'_i|=0$ and $|J_i|=|J'_i|=0$ for all $i\in I$.

The Hopf algebra structure ($\Delta_+, \varepsilon_+, S_+$) (resp.  ($\Delta_-, \varepsilon_-, S_-$)) 
on $\widehat{{}'\mathfrak{f}^+}$ (resp. $\widehat{{}'\mathfrak{f}^-}$) is  given as follows.
\allowdisplaybreaks
  \begin{eqnarray*}
  & \Delta_+(E_i)=E_i\otimes J_i+ K_i\otimes E_i,\quad &\Delta_+(K_i^{\pm 1})= K_i^{\pm 1} \otimes K_i^{\pm 1},\vspace{6pt}\\
  & \Delta_-(F_i)=J'_i \otimes F_i+ F_i\otimes K'_i, \quad & \Delta_-(K'^{\pm 1}_i)= K'^{\pm 1}_i \otimes K'^{\pm 1}_i, \vspace{6pt}\\
  &\Delta_+(J_i^{\pm 1})=J_i^{\pm 1}\otimes J_i^{\pm 1}, \quad & \Delta_-(J'^{\pm 1}_i)=J'^{\pm 1}_i\otimes J'^{\pm 1}_i,\\
  & \varepsilon_+(K_i^{\pm 1})=\varepsilon_+(J_i^{\pm 1})=1,\quad & \varepsilon_+(E_i)=0,\vspace{6pt}\\
  &\varepsilon_-(K'^{\pm 1}_i)=\varepsilon_-(J'^{\pm 1}_i)=1,\quad &\varepsilon_-(F_i)=0,\\
  & S_+(E_i)=-K_i^{-1}E_iJ_i^{-1},\quad & S_+(K_i^{\pm 1})=K_i^{\mp 1},\quad S_+(J_i^{\pm 1})=J_i^{\mp 1},\vspace{6pt}\\
  & S_-(F_i)=-J_i'^{-1}F_iK'^{-1}_i,\quad & S_-(K'^{\pm 1}_i)=K'^{\mp 1}_i,\quad S_-(J'^{\pm 1}_i)=J_i'^{^{\mp 1}}.
  \end{eqnarray*}

\subsection{Skew Hopf pairing}

For any $\nu\in \mbb Z[I]$ and $j\in I$, we define $\mbb F $-linear maps ${}_j\mcal S$, $\mcal J_{\nu}$ and 
${\mcal K}_{\nu}$ from $\wh{{}'\mrk{f}^-}$ to itself by
\allowdisplaybreaks
\begin{equation}\label{js}
  \begin{split}
&{\mcal K}_{\nu}(K'_{\tau_1}xJ'_{\tau_2})
=\langle \nu,\tau_1\rangle\langle \nu,|x|\rangle
K'_{\tau_1}xJ'_{\tau_2},\\
&\mcal J_{\nu}(K'_{\tau_1}xJ'_{\tau_2})=\xi(\tau_1,\nu)\xi(|x|,\nu)K'_{\tau_1}xJ'_{\tau_2},\quad
\forall \tau_1,\tau_2\in \mbb Z[I], x\in {}'\!\mrk f^-,\\
&{}_j\mcal S(K'_{\tau_1}J'_{\tau_2})=0,\quad {}_j\mcal S(F_i)=\delta_{ij}\xi(i,j),\quad {\rm and}\\
&{}_j\mcal S(xy)={}_j\mcal S(x)\mcal J_j(y)+{\mcal K}_j(x){}_j\mcal S(y).
  \end{split}
\end{equation}
It is clear that ${\mcal K}_{\nu}$ and $\mcal J_{\nu}$ are well defined.
We now check that ${}_j\mcal S$ is also well defined.
The commutative relations on $\mathfrak{H}'$ are trivial.
By the definition of ${}_l\mcal S$, for any $i, j, l\in I$, we have
\begin{equation*}
  \begin{split}
    &{}_l\mcal S(K_i'F_j)=\delta_{jl}\langle j,i\rangle \xi(j,l)K'_i=\langle j,i\rangle\xi(i,j)^{-1}\; {}_l\mcal S(F_j K_i'),\\
    & {}_l\mcal S(J_i'F_j)=\delta_{jl}\xi(j,l)J'_i=\,{}_l\mcal S(F_j J_i').
  \end{split}
\end{equation*}
This shows ${}_l\mcal S$ is well defined.
We note that $|{}_l\mcal S(x)|=|x|-l$ if $|x|_l\neq 0$.
\begin{lem}\label{lemrho^+}
The assignment $E_j\mapsto (v_j^{-1}-v_j)^{-1}{}_j\mcal S, K_{\nu}\mapsto {\mcal K}_{\nu}$ and $J_{\nu}\mapsto \mcal J_{\nu}$ 
defines an algebra anti-homomorphism
$\rho^+: \wh{{}'\mrk{f}^+}\rightarrow \Hom_{\mbb F}(\wh{{}'\mrk{f}^-}, \wh{{}'\mrk{f}^-})$.
\end{lem}
\begin{proof} For any $i, j\in I$, $\tau_1,\tau_2\in \mbb Z[I]$ and $x\in {}'\mrk{f}^-$, we have
\begin{equation*}
\begin{split}
&{}_j\mcal S ({\mcal K}_i(K'_{\tau_1}xJ'_{\tau_2}))=
  \langle i,\tau_1\rangle \langle i,|x|\rangle
  \langle j,\tau_1\rangle
  K'_{\tau_1}\,{}_j\mcal S(x)J'_{\tau_2},\quad {\rm and}\\
  &{\mcal K}_i({}_j\mcal S(K'_{\tau_1}xJ'_{\tau_2}))=
  \langle i,\tau_1\rangle \langle i,|x|-j\rangle
  \langle j,\tau_1\rangle
  K'_{\tau_1}\,{}_j\mcal S(x)J'_{\tau_2}.
\end{split}
\end{equation*}
This shows that $\rho^+(K_iE_j)=\langle i,j\rangle\rho^+(E_jK_i)$.

Similarly, for any $i, j\in I$, $\tau_1,\tau_2\in \mbb Z[I]$ and $x\in {}'\mrk{f}^-$, we have
\begin{equation*}
\begin{split}
&{}_j\mcal S (\mcal J_i(K'_{\tau_1}xJ'_{\tau_2}))=
  \langle j,\tau_1\rangle  \xi(|x|,i)
  \xi(\tau_1,i)
  K'_{\tau_1}\ {}_j\mcal S(x)J_{\tau_2},\quad {\rm and}\\
  &\mcal J_i(\,{}_j\mcal S(K'_{\tau_1}xJ'_{\tau_2}))=
  \langle j,\tau_1\rangle  \xi(|x|-j,i)
  \xi(\tau_1,i)
  K'_{\tau_1}\,{}_j\mcal S(x)J'_{\tau_2}.
\end{split}
\end{equation*}
This shows that $\rho^+(J_iE_j)=\xi(j,i)\rho^+(E_jJ_i)$.
The other defining relations for $\wh{{}'\mrk{f}^+}$ are straightforward.
\end{proof}

We define a bilinear form $\phi: \wh{{}'\mathfrak{f}^+}\times \wh{{}'\mathfrak{f}^-} \rightarrow \mbb F$ by
\begin{equation}\label{hopfpair}
  \phi(x,y)=\varepsilon_-(\rho^+(x)y),\ \forall x \in \wh{{}'\mathfrak{f}^+}, y \in \wh{{}'\mathfrak{f}^-}.
\end{equation}
The bilinear form $\phi$ is well defined. It is clear that $\phi(x,y)=0$ if $x, y $ are homogenous with different gradings.

Moreover,
for any $\nu_1,\nu_2, \tau_1,\tau_2 \in \mbb Z[I]$ and any $x\in {}'\!\mrk f^+, y\in {}'\!\mrk f^-$, we have
\begin{equation}
  \label{phikxy}
  \begin{split}
  &\phi(K_{\nu_1}xJ_{\nu_2}, K'_{\tau_1}yJ'_{\tau_2})
  =\langle\nu_1,\tau_1\rangle
  \langle \nu_1,|y|\rangle
  \langle |x|, \tau_1\rangle
  \xi(\tau_1,\nu_2)\phi(x,y).
  \end{split}
\end{equation}
This can be shown as follows.
\allowdisplaybreaks
\begin{equation*}
  \begin{split}
  &\phi(K_{\nu_1}xJ_{\nu_2}, K'_{\tau_1}yJ'_{\tau_2})
  =\varepsilon_-(\mcal J_{\nu_2}({}_x\!\mcal S ({\mcal K}_{\nu_1}(K'_{\tau_1}yJ'_{\tau_2}))))\\
  =&\langle\nu_1,\tau_1\rangle
  \langle \nu_1,|y|\rangle
  \langle |x|, \tau_1\rangle
  \varepsilon_-(\mcal J_{\nu_2}(K'_{\tau_1}{}_x\!\mcal S(y)J'_{\tau_2}))\\
  =&\langle\nu_1,\tau_1\rangle
  \langle \nu_1,|y|\rangle
  \langle |x|, \tau_1\rangle
  \xi(\tau_1,\nu_2)\phi(x,y),
  \end{split}
\end{equation*}
where ${}_x\!\mcal S=\rho^+(x)$. It's clear $|{}_x\!\mcal S(y)|=0$.

\begin{prop}\label{prophopfpair}
  The bilinear form $\phi$ defined in \eqref{hopfpair} is a skew Hopf pairing.
\end{prop}

\begin{proof}
  Part (a) in Definition \ref{def1} follows directly from the definition of $\phi$.

  We show (b) in Definition \ref{def1}.
For any $\nu_1, \nu_2, \mu_1, \mu_2, \tau_1, \tau_2\in \mbb Z[I]$ and any $x\in {}'\!\mrk f^+, y',y''\in {}'\!\mrk f^-$,
by \eqref{phikxy}, we have
\begin{equation}
\label{phikxy1y2}
\begin{split}
  &\quad\quad\quad\quad\quad\phi(K_{\nu_1}xJ_{\nu_2},K'_{\mu_1}y'J'_{\mu_2} K'_{\tau_1}y''J'_{\tau_2})=\\
  &\langle |y'|,\tau_1\rangle^{-1}\xi(\tau_1,|y'|)\langle \nu_1 +|x|,\mu_1+\tau_1\rangle \langle \nu_1,|x|\rangle
  \xi(\mu_1+\tau_1,\nu_2)\phi(x, y'y'').
  \end{split}
\end{equation}

By the definition of $\Delta_+$, we may write $\Delta_+(x)=\sum K_{|x_2|}x_1\otimes x_2J_{|x_1|}$ 
with $x_1, x_2 \in {}'\mrk f^+$. By \eqref{phikxy}, we have
\begin{equation}
\label{phikxy1timesy2}
\begin{split}
  &\phi(\Delta_+(K_{\nu_1}xJ_{\nu_2}),K'_{\mu_1}y'J'_{\mu_2}\otimes K'_{\tau_1}y''J'_{\tau_2})=\langle |x_2|,\tau_1\rangle \xi(\tau_1,|x_1|)\\
  &\langle |x|,\mu_1\rangle \langle \nu_1,\mu_1+\tau_1+|x|\rangle
  \xi(\mu_1+\tau_1,\nu_2)\phi(\Delta_+(x), y'\otimes y'').
 \end{split}
\end{equation}
By comparing \eqref{phikxy1y2} and \eqref{phikxy1timesy2} with $|x_1|=|y'|, |x_2|=|y''|$, (b) is reduced to
\begin{equation}
  \label{phixy1y2}
  \phi(x, y'y'')=\phi(\Delta_+(x),y'\otimes y''),
\end{equation}
where $\phi(x'\otimes x'', y'\otimes y'')=\phi(x', y')\phi(x'',y''),$ for all $x, x' ,x''\in {}'\mrk f^+, y',y''\in {}'\mrk f^-$.

Let us simply write $\Delta_+(x)=\sum x_1\otimes x_2$ with $x_1,x_2\in  \wh{{}'\mathfrak{f}^+}$, not necessarily in ${}'\!\mrk f^{+}$.
By the definition of $\phi$ , \eqref{phixy1y2} is equivalent to
\begin{equation}
  \label{rho+}
  \rho^+(x)(y'y'')=\textstyle \sum \rho^+(x_1)(y')\rho^+(x_2)(y'').
\end{equation}
We show \eqref{rho+} by induction on $|x|$.
If $x=E_i$, \eqref{rho+} can be directly shown by the definition of $\phi$.
Assume that \eqref{rho+} holds for $x',x''$.
We write
$\Delta_+(x')=\sum x_1'\otimes x_2'$ and  $\Delta_+(x'')=\sum x_1''\otimes x_2''$.
Then $\Delta_+(x'x'')=\sum  x_1'x_1''\otimes x_2'x_2''$ and
\begin{equation*}
\begin{split}
  \rho^+(x'x'')(y'y'')
  &=\rho^+(x'')\rho^+(x')(y'y'')\\
 &=\textstyle\sum \rho^+(x''_1)\rho^+(x'_1)(y')\rho^+(x''_2)\rho^+(x'_2)(y'')\\
& =\textstyle \sum \rho^+(x'_1x''_1)(y') \rho^+(x'_2x''_2)(y'').
\end{split}
\end{equation*}
This proves \eqref{rho+} and, therefore, (b).

We show (c) in Definition \ref{def1}.
By using \eqref{phikxy} and a similar argument as we did for (b),
(c) is reduced to
\begin{equation}
  \label{phix1x2y}
 \phi(x'x'', y)=\phi(x'\otimes x'',\Delta_-^{\rm op}(y)),\quad \forall x',x''\in {}'\mrk f^+, y\in {}'\mrk f^-.
\end{equation}
We show \eqref{phix1x2y} by induction on $|y|$.
The case that $y=F_i$ is straightforward.
Assume that \eqref{phix1x2y} holds for $y',y''$ in the second component.
We write
$\Delta^{\rm op}_-(y')=\sum y_1'\otimes y_2'$ and  $\Delta^{\rm op}_-(y'')=\sum y_1''\otimes y_2''$.
By \eqref{phixy1y2}, we have
\begin{equation*}
  \begin{split}
    \phi(x'x'', y'y'')&=\phi(\Delta_+(x'x''),y'\otimes y'')=\textstyle \sum\phi(x_1'x_1'', y')\phi(x_2'x_2'', y'')\\
    &=\textstyle \sum\phi(x_1'\otimes x_1'', \Delta^{\rm op}_-(y'))\phi(x_2'\otimes x_2'',
    \Delta^{\rm op}_-(y''))\\
    &=\textstyle \sum\phi(x_1', y_1')\phi(x_2',y_1'')\phi(x_1'', y_2')\phi(x_2'',y_2'')\\
    &=\textstyle \sum\phi(x',y'_1y''_1)\phi(x'',y'_2y''_2)\\
    &=\textstyle \sum \phi(x'\otimes x'', y'_1y''_1\otimes y'_2y''_2)
    =\phi(x'\otimes x'', \Delta_-^{\rm op}(y'y'')).
  \end{split}
\end{equation*}
This proves (c).

We now show (d) in Definition \ref{def1}.
By definition of $S_+$, for any $x\in {}'\!\mrk f^+$, $S_+'(x)=K_{-|x|}x'J_{-|x|}$ with $x'\in  {}'\!\mrk f^+$ 
and $|x'|=|x|$, where $K_{-|x|}=K_{|x|}^{-1}$ and $J_{-|x|}=J_{|x|}^{-1}$.
Since $S_-'^{-1}(F_i)=-K_i'^{-1}F_iJ_i'^{-1}$, for any $y\in {}'\!\mrk f^-$, we can write 
$S_-'^{-1}(y)=K'_{-|y|}y'J'_{-|y|}$ with $y'\in  {}'\!\mrk f^-$ and $|y'|=|y|$.
By \eqref{phikxy}, for any $\nu_1, \nu_2, \tau_1, \tau_2 \in \mbb Z[I]$ and any $x\in  {}'\!\mrk f^+, y\in {}'\!\mrk f^-$, we have
\begin{equation}
  \label{phiskxy}
  \begin{split}
  & \phi(S_+(K_{\nu_1}xJ_{\nu_2}), K'_{\tau_1}yJ'_{\tau_2})\\
  &=\langle-\nu_1,\tau_1\rangle\langle\nu_1,|x|\rangle\langle-\nu_1,|y|\rangle\xi(|x|,\nu_2)^{-1}\xi(\tau_1,-\nu_2-|x|)\phi(S_+(x), y);\\
  & \phi(K_{\nu_1}xJ_{\nu_2}, S_-'^{-1}(K'_{\tau_1}yJ'_{\tau_2}))\\
  &=\langle |y|,\tau_1\rangle\langle|x|,-\tau_1\rangle\langle\nu_1,-\tau_1\rangle \xi(-\tau_1-|y|,\nu_2)\xi(\tau_1,|y|)^{-1}\phi(x, S_-'^{-1}(y)).
  \end{split}
\end{equation}
By using $|x|=|y|$ and comparing two identities in \eqref{phiskxy}, (d) is reduced to
\begin{equation}\label{phisxy}
  \phi(S_+(x), y)=\phi(x, S_-'^{-1}(y)), \quad \forall x\in  {}'\!\mrk f^+, y\in {}'\!\mrk f ^-.
\end{equation}
We show \eqref{phisxy} by induction on $|x|$.
The case that $x=E_i$ is straightforward.
Assume that \eqref{phisxy} holds for $x', y'$ and $x'',y''$, respectively.
We write $\Delta_-(y')=\sum y'_1\otimes y'_2$ and $\Delta_-(y'')=\sum y''_1\otimes y''_2$.

Recall that $\Delta_-^{\rm op}S_-'^{-1}=\tau_{12}(S_-'^{-1}\otimes S_-'^{-1})\Delta_-^{\rm op}$, 
where $\tau_{12}$ is the operator interchanging the first and second component of the tensors.
By using this identity, we have
\begin{equation*}
  \begin{split}
    \phi(S_+(x'x''), y'y'')&=\phi(S_+(x'')\otimes S_+(x'), \Delta_-^{\rm op}(y'y''))\\
    &=\textstyle \sum \phi(x'', S_-'^{-1}(y_2'y_2''))\phi(x', S_-'^{-1}(y_1'y_1''))\\
   & =\phi(x'\otimes x'', \tau_{12}(S_-'^{-1}\otimes S_-'^{-1})\Delta_-^{\rm op}(y'y''))\\
    &=\phi(x'\otimes x'', \Delta_-^{\rm op}S_-'^{-1}(y'y''))\\
    &=\phi(x'x'',S_-'^{-1}(y'y'')).
  \end{split}
\end{equation*}
This finishes the proof.
\end{proof}

Recall $D_{ij}$ from \eqref{Dij}, denoted by $\widehat{\mrk J^+}$ (resp. $\widehat{\mrk J^-}$) the ideal of  
$\widehat{{}'\mrk f ^+}$ (resp. $\widehat{{}'\mrk f ^-}$) generated by $\iota^+(D_{ij} )$ (resp. $\iota^-(D_{ij} ))$.
Let $$\widehat{\mrk f }^+=\widehat{{}'\mrk f ^+}/\widehat{\mrk J^+}\quad \rm{and}\quad
\widehat{\mrk f }^-=\widehat{{}'\mrk f ^-}/\widehat{\mrk J^-}.$$

In the rest of this subsection, we shall show that $\phi$ induces a well-defined skew-Hopf pairing 
$\widehat{\mrk f }^+\times\widehat{\mrk f }^-\rightarrow \mathbb{F}$.

Similarly, for any $\nu,\tau_1,\tau_2\in \mbb Z[I]$ and any $i,j\in I$, 
we define $\mbb F$-linear maps $\mcal S\!_j$, $\mcal J'_{\nu}$ and ${\mcal K}'_{\nu}$ from $\wh{{}'\mrk{f}^+}$ to itself by
\allowdisplaybreaks
\begin{equation}\label{sj}
  \begin{split}
&{\mcal K}'_{\nu}(K_{\tau_1}xJ_{\tau_2})
=\langle \tau_1,\nu\rangle\langle |x|,\nu\rangle\xi(\nu,|x|)^{-1}\xi(\nu,\tau_2)
K_{\tau_1}xJ_{\tau_2},\\
&\mcal J'_{\nu}(K_{\tau_1}xJ_{\tau_2})=K_{\tau_1}xJ_{\tau_2},\quad
\forall \tau_1,\tau_2\in \mbb Z[I], x\in {}'\mrk f ^+,\\
&\mcal S\!_j(K_{\tau_1}J_{\tau_2})=0,\;\mcal S\!_j(E_i)=\delta_{ij},\ {\rm and}
\ \mcal S\!_j(xy)=\mcal S\!_j(x)y+{\mcal K}'_j(x)\mcal S\!_j(y).
  \end{split}
\end{equation}
A direct calculation shows that these maps are well defined.

Let $\mathcal{G}(i)=\xi(i,i)^{1/2}, \forall i \in I.$ In general, 
there are two different choices for $\mathcal{G}(i)$, we shall choose one such that
\begin{equation}\label{eqg}
\begin{split}
\mathcal{G}(\nu_1)\mathcal{G}(\nu_2)\xi(\nu_2,\nu_1)&=\mathcal{G}(\nu_1+\nu_2),\quad \forall \nu_1,\nu_2\in \mbb N[I].
\end{split}
\end{equation}
Let $\rho^-: \wh{{}'\mrk{f}^-}\rightarrow \Hom_{\mbb F}(\wh{{}'\mrk{f}^+}, \wh{{}'\mrk{f}^+})$ be the
map defined by
\begin{equation*}
F_j\mapsto \frac{\mathcal{G}(j)}{v_j^{-1}-v_j}\mcal S_j,\quad K'_{\nu}\mapsto {\mcal K}'_{\nu}\quad {\rm and}
\quad   J'_{\nu}\mapsto \mcal J'_{\nu}.
\end{equation*}
Then we have the following lemma.
\begin{lem}\label{lemrho^-}
 {\rm (i)} $\rho^-$ is an algebra anti-homomorphism.

 {\rm (ii)} For any $y\in {}'\mrk f ^-$, $x',x''\in {}'\mrk f ^+$,
we have
  $$\rho^-(y)(x'x'')=\textstyle \sum\rho^-(y_2)(x')\rho^-(y_1)(x''),$$
   where $\Delta_-(y)=\sum y_1\otimes y_2$.
\end{lem}
The proof is similar to those for Lemma \ref{lemrho^+} and \eqref{rho+}.

We now define a new bilinear form $\phi': \wh{{}'\mathfrak{f}^+}\times \wh{{}'\mathfrak{f}^-} \rightarrow \mbb F$ by
\begin{equation*}
  \phi'(x,y)=\varepsilon_+(\rho^-(y)x),\  \forall x \in \wh{{}'\mathfrak{f}^+}, y \in \wh{{}'\mathfrak{f}^-}.
\end{equation*}
The bilinear form $\phi'$ is well defined.

By Lemma \ref{lemrho^-}(ii), we have
\begin{equation*}
    \phi'(x'x'',y)=\phi'(x'\otimes x'', \Delta_-^{\rm op}(y)),
\end{equation*}
where $\phi'(x'\otimes x'', y'\otimes y'')=\phi'(x', y')\phi'(x'',y'')$, 
for all $x',x''\in  {}'\mrk f ^+,y, y',y'' \in {}'\mathfrak{f}^-$.

Moreover, for any $\nu_1,\nu_2, \tau_1,\tau_2 \in \mbb Z[I]$ and $x\in {}'\mrk f ^+, y\in {}'\mrk f ^-$, we have
\begin{equation}
  \label{phi'kxy}
  \begin{split}
  &\quad\quad\quad\quad\phi'(K_{\nu_1}xJ_{\nu_2}, K'_{\tau_1}yJ'_{\tau_2})=\\
  &\langle\nu_1,\tau_1\rangle
  \langle \nu_1,|y|\rangle
  \langle |x|, \tau_1\rangle\xi(\tau_1,|x|)^{-1}
  \xi(\tau_1,\nu_2)\phi'(x,y).
  \end{split}
\end{equation}
This equality can be directly checked by the definition of $\phi'$. 
We notice that the coefficient on the right hand side of \eqref{phi'kxy} is different from the one in \eqref{phikxy}.
The difference is $\xi(\tau_1,|x|)^{-1}$.
\begin{lem}
  \label{lemg}
 For any homogenous elements $x\in {}'\!\mrk f ^+$ and $y\in {}'\!\mrk f ^-$, we have
  $\phi(x,y)=\mathcal{G}(|x|)\phi'(x,y)$.
\end{lem}
\begin{proof}
  We show it by induction on $|x|$.
  If $|x|=i$, we have
  \begin{equation*}
    \phi(E_i, F_j)=\delta_{ij}\frac{\xi(i,i)}{v_i^{-1}-v_i}=\mathcal{G}(i)\phi'(E_i, F_j).
  \end{equation*}
  Assume that the lemma holds for $x',x''\in  {}'\mrk f ^+$ in the first component.
  For any $y\in {}'\mrk f ^-$, we write $\Delta_-^{\rm op}(y)=\sum K'_{|y_1|}y_2\otimes y_1J'_{|y_2|}$ with $y_1, y_2\in {}'\mrk f ^-$.
  Then we have
  \begin{equation*}
    \begin{split}
      &\phi(x'x'',y)=\textstyle \sum \phi(x', K'_{|y_1|}y_2)\phi(x'', y_1J'_{|y_2|})\\
      =&\textstyle \sum\langle |x'|,|y_1|\rangle
     \mathcal{G}(|x'|)\mathcal{G}(|x''|)\phi'(x',y_2)\phi'(x'',y_1)\\
      =&\textstyle \sum \mathcal{G}(|x'|)\mathcal{G}(|x''|)\xi(|y_1|,|x'|)
      \phi'(x', K'_{|y_1|}y_2)\phi'(x'', y_1J'_{|y_2|})\\
      =&\mathcal{G}(|x'|+|x''|)\phi'(x'x'',y).
    \end{split}
  \end{equation*}
  Here we use the fact that $|x''|=|y_1|$.
  This finishes the proof.
\end{proof}

By Proposition \ref{propserrerelation}, we have
\begin{equation}\label{lr}
{}_lr(D_{ij})=0=r_l(D_{ij}),\quad \forall i\neq j,\ l\in I.
\end{equation}

By comparing the definition of ${}_lr$  with that of $ \mathcal{S}_l$ in \eqref{sj} , we have
\begin{equation}\label{iota}
\mathcal{S}_l |_{'\mathfrak{f}^+}\circ \iota^+=\iota^+ \circ {}_lr .
\end{equation}

\begin{lem}\label{lemis}
  For any homogenous elements $x\in {}'\mrk f $ and $i\in I$, we have
  $${}_i\mcal S |_{'\mathfrak{f}^-}\circ\iota^-(x)=\xi(|x^-|,i)\iota^-\circ r_i (x),$$
  where $x^-=\iota^-(x).$
\end{lem}

\begin{proof}
  We show it by induction on $|x^-|$. Recall the definition of ${}_i\mathcal{S}$ from \eqref{js}, 
  the case that $|x^-|=j$ is trivial. Assume that the lemma holds for $x,y$.
 Then we have
 \begin{equation*}
   \begin{split}
     &{}_i\mcal S |_{'\mathfrak{f}^-}\circ\iota^-(xy)=\xi(|x^-|,i){}_i\mcal S |_{'\mathfrak{f}^-}(y^-)x^-+
              \langle i, |y^-|\rangle y^- {}_i\mcal S |_{'\mathfrak{f}^-}(x^-)\\
     =&\xi(|x^-|+|y^-|,i)\,\iota^-\circ r_i(y)x^-+\langle i, |y^-|\rangle\xi(|x^-|,i) y^-\, \iota^-\circ r_i (x)\\
     =&\xi(|x^-|+|y^-|,i)\,(\iota^-\circ r_i (y)x^-+\langle i, |y^-|\rangle\xi(|y^-|,i)^{-1} y^-\, \iota^-\circ r_i (x))\\
     =&\xi(|x^-|+|y^-|,i)\,\iota^-\circ r_i (xy).
   \end{split}
 \end{equation*}
 Lemma is proved.
\end{proof}

By \eqref{lr},  \eqref{iota} and Lemma \ref{lemis}, we have
\begin{equation}\label{lssl}
\begin{split}
 &\mathcal{S}_l (\iota^+(D_{ij} ))=\iota^+ \circ {}_lr (D_{ij} )=0, \\
 &{}_l\mathcal{S} (\iota^-(D_{ij} ))=\xi(|D_{ij}^-|,l)\iota^- \circ r_l (D_{ij} )=0, \quad \forall i \neq j,\ l \in I.
\end{split}
\end{equation}

\begin{cor}
  The bilinear form $\phi$ induces a well defined skew-Hopf pairing  
  $\widehat{\mathfrak{f}}^+\times \widehat{\mathfrak{f}}^-\rightarrow \mbb{F}$.
\end{cor}

\begin{proof}
  By the definition of $\ \widehat{\mrk f }^+$ and $\ \widehat{\mrk f }^-$,
  it is enough to show that $\widehat{\mrk J^+}$ and  $\widehat{\mrk J^-}$ are in the radical of $\phi$,
  i.e., $\phi(\iota^+(D_{ij} ), y)=0=\phi(x, \iota^-(D_{ij} ))$ for any $i\neq j$, 
  $ x\in {}'\!\widehat{\mrk f }^+$ and $y\in {}'\!\widehat{\mrk f }^-$.
  By \eqref{phikxy}, we may assume that $ x\in {}'\!\mrk f ^+$ and $y\in {}'\!\mrk f ^-$.
  If $|x|=|y|=0$, there is nothing to show.
  We now assume that $|x|$ and $|y|$ are not equal to zero.
  Then we can write $x=E_lx'$ for some $l$.
  By the definition of $\phi$ and \eqref{lssl}, we have
  \begin{equation*}
    \phi(x, \iota^-(D_{ij}))=\varepsilon_-(\rho^+(x'){}_l\mcal S(\iota^-(D_{ij} )))=0, \quad \forall i\neq j.
  \end{equation*}
  Similarly, $\phi'(\iota^+(D_{ij},y))=0$ for any $i\neq j$.
  By Lemma $\ref{lemg}$ , we have
  \begin{equation*}
       \phi(\iota^+(D_{ij}),y)=\mathcal{G}(|D_{ij}|)\phi'( \iota^+(D_{ij}),y)=0,  \quad \forall i\neq j.
  \end{equation*}
 This finishes the proof.
\end{proof}

By Proposition \ref{propdrinfelddouble}, there is a Hopf algebra structure on 
$\wh{\mrk{f}}^+\otimes \wh{\mrk{f}}^-$, denoted this Hopf algebra by $\mbf U_{\beta,\xi}$.

\subsection{Algebraic presentation of $\mbf U_{\beta,\xi}$}\label{sec3.4}

It is clear that $\mbf U_{\beta,\xi}$ is generated by $E_i, F_i,J_i^{\pm 1}, J_i'^{\pm 1}, K_i^{\pm 1}$ 
and $ K'^{\pm 1}_i$ for all $i\in I$. We now calculate defining relations of $\mbf U_{\beta,\xi}$ in these generators.

Since $\wh{\mrk f }^+$ and $\wh{\mrk f }^-$ can be thought as subalgebras of $\mbf U_{\beta,\xi}$, 
the algebra $\mbf U_{\beta,\xi}$ carries on all defining relations of $\wh{\mrk f }^+$ and $\wh{\mrk f }^-$.
It is easy to show that $J_i^{\pm 1}, J_i'^{\pm 1}, K_i^{\pm 1}$ and $  K'^{\pm 1}_i$ commute with each other.

Since $\Delta^2_-(J_i')=J'_i\otimes J'_i\otimes J'_i$,\  $\Delta^2_-(K_i')=K'_i\otimes K'_i\otimes K'_i$ and
\begin{equation}
  \label{deltaE}
  \Delta_+^2(E_i)=E_i\otimes J_i\otimes J_i+K_i\otimes E_i\otimes J_i+ K_i\otimes K_i\otimes E_i,
\end{equation}
by (iii) in Proposition \ref{propdrinfelddouble}, we have
\begin{equation*}
  \begin{split}
    K_i'E_j=&\ \phi(K_j, K_i'^{-1})E_jK'_i\phi(J_j, K_i')=\langle j,i\rangle^{-1}\xi(i,j)E_jK_i',\ {\rm and}\\
    &J_i'E_j=\phi(K_j, J_i'^{-1})E_jJ_i'\phi(J_j, J_i')=E_jJ_i'.
  \end{split}
\end{equation*}
Since $\Delta^2_+(J_i)=J_i\otimes J_i\otimes J_i$,\  $\Delta^2_+(K_i)=K_i\otimes K_i\otimes K_i$ and
\begin{equation}
  \label{deltaF}
  \Delta_-^2(F_i)=J'_i\otimes J'_i\otimes F_i +J'_i\otimes F_i\otimes K_i'+ F_i\otimes K_i'\otimes K'_i,
\end{equation}
 we have
 \begin{equation*}
  \begin{split}
    F_j K_i&\ =\phi(K_i, J_j'^{-1})K_i F_j\phi(K_i, K_j')=\langle i,j\rangle K_i F_j,\ {\rm and}\\
    &F_j J_i=\phi(J_i, J_i'^{-1})J_iF_j\phi(J_i, K_j')=\xi(j,i)J_iF_j.
  \end{split}
\end{equation*}
 Similarly, by \eqref{deltaE} and \eqref{deltaF}, we have
 \begin{equation*}
   \begin{split}
   F_jE_i =\phi(E_i,-&J_j'^{-1}F_jK_j'^{-1})J_iK_j'\phi(J_i,K_j')+E_iF_j\phi(J_i,K_j')
     +K_iJ_j'\phi(E_i,F_j)\\
     &=\xi(j,i)E_iF_j+\delta_{ij}\frac{\xi(j,i)}{v_i^{-1}-v_i}(K_iJ_j'-J_iK_j').
   \end{split}
 \end{equation*}

Summarizing up, we have the following presentation of $\mbf U_{\beta,\xi}$ generated by symbols 
$E_i, F_i, J_i^{\pm 1},J_i'^{\pm 1}, K_i^{\pm 1}, K_i'^{\pm 1},$ $\forall i\in I$, and subjects to the following relations.
\allowdisplaybreaks
\begin{eqnarray*}
  (R1) & &J_i^{\pm 1}, J_i'^{\pm 1}, K_i^{\pm 1}\ \text{and}\  K'^{\pm 1}_i\ \text{commute with each other},\\
 & &    K_i^{\pm 1}K_i^{\mp 1}=K'^{\pm 1}_iK'^{\mp 1}_i=J_i^{\pm 1}J_i^{\mp 1}=J_i'^{\pm 1}J_i'^{\mp 1}=1.\\
  (R2)& &K_iE_j=\langle i,j\rangle E_jK_i,\quad K'_iE_j=\langle j, i\rangle^{-1}\xi(i,j) E_jK'_i,\\
     & &K_iF_j=\langle i, j\rangle^{-1} F_jK_i,\quad K'_iF_j=\langle j,i\rangle \xi(i,j)^{-1} F_jK'_i,\\
     & &J_iE_j=\xi(j,i)E_jJ_i,\quad  J'_iE_j=E_jJ'_i,\\
     & &J_iF_j=\xi(j,i)^{-1}F_jJ_i,\ \ J'_iF_j=F_jJ'_i.\\
  (R3)& &E_iF_j-\xi(j,i)^{-1}F_j E_i=\delta_{ij}\frac{K_iJ'_i-J_iK'_i}{v_i-v^{-1}_i}.\\
  (R4) & &\sum_{k+k'=1-a_{ij}}(-1)^k(\beta(i,j))^{-k}E_i^{(k)}E_j E_i^{(k')}=0,\quad {\rm if}\ i\not =j,\\
  & &\sum_{k+k'=1-a_{ij}}(-1)^k(\beta(i,j))^{-k}F_i^{(k')}F_j F_i^{(k)}=0,\quad {\rm if}\ i\not =j.\\
\end{eqnarray*}

We note that $\mbf U_{\beta,\xi}$ admits a triangle decomposition $\mbf U_{\beta,\xi} \simeq \mbf U^+ \otimes \mbf U^0 \otimes \mbf U^-$
by repeating the argument for ordinary quantum groups word by word,
where $\mbf U^+$ (resp. $\mbf U^0$ or $\mbf U^-$) is a subalgebra of $\mbf U_{\beta,\xi}$ 
generated by $E_i$ (resp. $J_i^{\pm 1},J_i'^{\pm 1}, K_i^{\pm 1}, K_i'^{\pm 1}$ or $F_i$) for all $i \in I$.

\subsection{Hopf algebra structure on $\mbf U_{\beta,\xi}$}\label{sec3.5}

By Proposition \ref{propdrinfelddouble}, $\mbf U_{\beta,\xi}$ has a Hopf algebra structure with 
the comultiplication $\Delta$, the counit $\varepsilon$ and the antipode $S$ given as follows.
\allowdisplaybreaks
  \begin{eqnarray*}
  & \Delta(E_i)=E_i\otimes J_i+ K_i\otimes E_i,\quad &\Delta(K_i^{\pm 1})= K_i^{\pm 1} \otimes K_i^{\pm 1},\vspace{6pt}\\
  & \Delta(F_i)=J'_i \otimes F_i+ F_i\otimes K'_i, \quad & \Delta(K'^{\pm 1}_i)= K'^{\pm 1}_i \otimes K'^{\pm 1}_i, \vspace{6pt}\\
  &\Delta(J_i^{\pm 1})=J_i^{\pm 1}\otimes J_i^{\pm 1}, \quad & \Delta(J'^{\pm 1}_i)=J'^{\pm 1}_i\otimes J'^{\pm 1}_i,\\
  & \varepsilon(K_i^{\pm 1})=\varepsilon(J_i^{\pm 1})=1,\quad & \varepsilon(E_i)=0,\vspace{6pt}\\
  &\varepsilon(K'^{\pm 1}_i)=\varepsilon(J'^{\pm 1}_i)=1,\quad &\varepsilon(F_i)=0,\\
  & S(E_i)=-K_i^{-1}E_iJ_i^{-1},\quad & S(K_i^{\pm 1})=K_i^{\mp 1},\quad S(J_i^{\pm 1})=J_i^{\mp 1},\vspace{6pt}\\
  & S(F_i)=-J_i'^{-1}F_iK'^{-1}_i,\quad & S(K'^{\pm 1}_i)=K'^{\mp 1}_i,\quad S(J'^{\pm 1}_i)=J_i'^{^{\mp 1}}.
  \end{eqnarray*}

\section{Specializations of $\mbf U_{\beta,\xi}$}\label{sec4}
In this section, we shall show that various quantum algebras in Section \ref{sec2.3} 
admit a Drinfeld double construction by specializing the parameters $\beta(i,j)$ and $\xi(i,j)$, for all $i, j \in I$.
\subsection{The entire two-parameter quantum algebras}
  In the case of the two-parameter quantum algebra (see Example \ref{extwoparameter}), we set
 \begin{equation*}
\beta(i,j)=t^{\Omega_{ji}-\Omega_{ij}}\quad{\rm and}\quad\xi(i,j)=1,\quad \forall i, j\in I.
\end{equation*}
Let $\mathcal{G}(\nu)=1$ for all $\nu\in \mbb N[I]$, then
 \eqref{eqg} holds.

Under these settings, we get the following presentation of an entire two-parameter quantum algebra $\mbf U_{v,t}$,
generated by symbols $E_i, F_i, K_i^{\pm 1}, K_i'^{\pm 1}, $
$J_i^{\pm 1}, J_i'^{\pm 1}$, $\forall i\in I$, and subjects to the following relations.
\allowdisplaybreaks
\begin{eqnarray*}
  (R_11) & &J_i^{\pm 1}, J_i'^{\pm 1}, K_i^{\pm 1}\ \text{and}\  K'^{\pm 1}_i\ \text{commute with each other},\\
 & &    K_i^{\pm 1}K_i^{\mp 1}=K'^{\pm 1}_iK'^{\mp 1}_i=J_i^{\pm 1}J_i^{\mp 1}=J_i'^{\pm 1}J_i'^{\mp 1}=1.\\
 (R_12)& &K_iE_jK^{-1}_i=v^{-i\cdot j}t^{\Omega_{ji}-\Omega_{ij}} E_j,\ \ K'_iE_jK'^{-1}_i=v^{i\cdot j}t^{\Omega_{ji}-\Omega_{ij}}E_j,\\
     & &K'_iF_jK'^{-1}_i=v^{-i\cdot j}t^{\Omega_{ij}-\Omega_{ji}} F_j,\ \ K_iF_jK^{-1}_i=v^{i\cdot j}t^{\Omega_{ij}-\Omega_{ji}}F_j,\\
     & &J_iE_jJ^{-1}_i=E_j,\ \ J'_iE_jJ'^{-1}_i=E_j,\ \ J_iF_jJ^{-1}_i=F_j,\ \ J'_iF_jJ'^{-1}_i=F_j.\\
  (R_13)& &E_iF_j-F_j E_i=\delta_{ij}\frac{K_iJ'_i-J_iK'_i}{v_i-v^{-1}_i}.\\
  (R_14)& &\sum_{k+k'=1-a_{ij}}(-1)^k\begin{bmatrix}1-a_{ij}\\k\end{bmatrix}_{v_i}
      t^{k(\Omega_{ij}-\Omega_{ji})}E_i^{k}E_jE_i^{k'}=0, \quad {\rm if}\ i\not =j,\\
  & &\sum_{k+k'=1-a_{ij}}(-1)^k\begin{bmatrix}1-a_{ij}\\k\end{bmatrix}_{v_i}
      t^{k(\Omega_{ij}-\Omega_{ji})}F_i^{k'}F_jF_i^{k}=0, \quad {\rm if}\ i\not =j.\\
\end{eqnarray*}

Let $U_{v,t}$ be the two-parameter quantum algebra in \cite{FL1} associated to $(I,\cdot)$.
\begin{prop}\label{proptwoparameter}
The map $\Phi:U_{v,t}\rightarrow\mbf U_{v,t}/\langle J_i-1, J'_i-1\rangle$ sending 
 $E_i\mapsto F_i$, $F_i\mapsto E_i$, $K_{i}\mapsto -K_i$ and $K'_i\mapsto -K'_i$ for all $i\in I$ is a $\mbb Q(v,t)$-algebra isomorphism.
\end{prop}

By Proposition \ref{proptwoparameter}, if we set $J_i=J'_i=1$ in Section \ref{sec3}, 
then the construction in Section \ref{sec3} provides a Drinfeld double construction of $U_{v,t}$. 
This recovers the results in \cite{BGH, BW}.

\subsection{The entire quantum superalgebras}
  In the case of the quantum superalgebra (see Example \ref{exsuper}), we set
\[\xi(i,j)=t^{2\mathcal{P}(i)\mathcal{P}(j)}\quad {\rm and}\quad \beta(i,j)=t^{i\cdot j}t^{2\mathcal{P}(i)\mathcal{P}(j)},\ \forall i, j\in I,\]
where $t=\mathbf{i}$.

Then, $\langle i, j\rangle=(v^{-1}t)^{i\cdot j}$.
We have
\begin{equation*}\label{xiij}
\xi(i,j)\xi(j,i)=1\quad {\rm and}\quad \langle i,j\rangle=\langle j,i\rangle, \  \forall i,j\in I.
\end{equation*}
Let $\mathcal{G}(\nu)=t^{(\mathcal{P}(\nu))^2}$ for all $\nu\in \mbb N[I]$, 
where $\mathcal{P}(\nu)=\sum_i\nu_i\mathcal{P}(i)$.
It's easy to check that $\mathcal{G}(\nu)$ satisfies \eqref{eqg}.
Under these settings, we have the following presentation of an entire quantum superalgebra 
$\mbf U_{v,\mathbf{i}}$ generated by symbols $E_i, F_i, J_i^{\pm 1},J_i'^{\pm 1}, K_i^{\pm 1}, K_i'^{\pm 1},$ $\forall i\in I$, 
and subjects to the following relations.
\allowdisplaybreaks
\begin{eqnarray*}
  (R_21) & &J_i^{\pm 1}, J_i'^{\pm 1}, K_i^{\pm 1}\ \text{and}\  K'^{\pm 1}_i\ \text{commute with each other},\\
 & &    K_i^{\pm 1}K_i^{\mp 1}=K'^{\pm 1}_iK'^{\mp 1}_i=J_i^{\pm 1}J_i^{\mp 1}=J_i'^{\pm 1}J_i'^{\mp 1}=1.\\
  (R_22) & &K_iE_j=v^{-i\cdot j}t^{i\cdot j} E_jK_i,\quad K'_iE_j=v^{i\cdot j}t^{-i\cdot j}t^{2\mathcal{P}(i)\mathcal{P}(j)} E_jK'_i,\\
     & &K_iF_j=v^{i\cdot j}t^{-i\cdot j} F_jK_i,\quad K'_iF_j=v^{-i\cdot j}t^{i\cdot j}t^{2\mathcal{P}(i)\mathcal{P}(j)} F_jK'_i,\\
     & &J_iE_j=t^{2\mathcal{P}(i)\mathcal{P}(j)}E_jJ_i,\quad  J'_iE_j=E_jJ'_i,\\
     & &J_iF_j=t^{2\mathcal{P}(i)\mathcal{P}(j)}F_jJ_i,\ \ J'_iF_j=F_jJ'_i.\\
  (R_23) & &E_iF_j-t^{2\mathcal{P}(i)\mathcal{P}(j)}F_j E_i=\delta_{ij}\frac{K_iJ'_i-J_iK'_i}{v_i-v^{-1}_i}.\\
  (R_24)  & &\sum_{k+k'=1-a_{ij}}(-1)^k\begin{bmatrix}1-a_{ij}\\k\end{bmatrix}_{v_i}
  t^{k i\cdot j+2k\mathcal{P}(i)\mathcal{P}(j)}E_i^{k}E_j E_i^{k'}=0,\quad {\rm if}\ i\not =j,\\
  & &\sum_{k+k'=1-a_{ij}}(-1)^k\begin{bmatrix}1-a_{ij}\\k\end{bmatrix}_{v_i}
  t^{k i\cdot j+2k\mathcal{P}(i)\mathcal{P}(j)}F_i^{k'}F_j F_i^{k}=0,\quad {\rm if}\ i\not =j.
\end{eqnarray*}
Let $U_{q,\mathbf{i}}$ be the quantum superalgebra defined in \cite{CHW1} associated to $(I,\cdot)$ which is a bar-consistent super Cartan datum.
Let $\mcal I$ be the ideal of $\mbf U_{v, \mbf i}$ generated by $\{K_iK_i'J_i-1,\ \forall i\in I\}$,
and $\mbf U_{v, \mbf i}^{sq}$ the subalgebra of $\mbf U_{v, \mbf i}/\mcal I$ generated by $E_i$, $F_i$, $J_i'^{\pm}$, $K_i^{\pm}$ and $(J_iK_i')^{\pm}$.
Then we have the following proposition.
\begin{prop}
 The map $\varsigma: U_{q,\mathbf{i}}\rightarrow \mbf U_{v, \mbf i}^{sq}$ sending
 $$q\mapsto v^{-1}t, \ E_i\mapsto E_i, \ F_i\mapsto F_i, \  J_i \mapsto t_i^2J_i',\ K_i\mapsto t_i^{-1}K_i, \ \forall\ i\in I$$
  is a $\mbb Q$-algebra isomorphism.
\end{prop}

\subsection{The entire  Multi-parameter quantum algebras}
  In the case of the multi-parameter quantum  algebra (see Example \ref{exmul}), we set
\[\beta(i,j)=v^{i\cdot j}q_{ij}\quad{\rm and}\quad\xi(i,j)=1,\quad \forall i, j\in I.\]
Let $\mathcal{G}(\nu)=1$ for all $\nu\in \mbb N[I]$, then \eqref{eqg} holds.

 Under these settings, we get the following presentation of an entire multi-parameter quantum algebra $\mbf U_{v,\mathbf{q}}$,
generated by symbols $E_i, F_i, K_i^{\pm 1}, K_i'^{\pm 1}, $
$J_i^{\pm 1}, J_i'^{\pm 1}$, $\forall i\in I$, and subjects to the following relations.
\allowdisplaybreaks
\begin{eqnarray*}
  (R_31)& &J_i^{\pm 1}, J_i'^{\pm 1}, K_i^{\pm 1}\ \text{and}\  K'^{\pm 1}_i\ \text{commute with each other},\\
 & &    K_i^{\pm 1}K_i^{\mp 1}=K'^{\pm 1}_iK'^{\mp 1}_i=J_i^{\pm 1}J_i^{\mp 1}=J_i'^{\pm 1}J_i'^{\mp 1}=1.\\
  (R_32)& &K_iE_jK^{-1}_i=q_{ij} E_j,\ \ K'_iE_jK'^{-1}_i=q_{ji}^{-1}E_j,\\
     & &K'_iF_jK'^{-1}_i=q_{ji}F_j,\ \ K_iF_jK^{-1}_i=q_{ij}^{-1}F_j,\\
     & &J_iE_jJ^{-1}_i=E_j,\ \ J'_iE_jJ'^{-1}_i=E_j,\ \ J_iF_jJ^{-1}_i=F_j,\ \ J'_iF_jJ'^{-1}_i=F_j. \\
  (R_33)& &E_iF_j-F_j E_i=\delta_{ij}\frac{K_iJ'_i-J_iK'_i}{v_i-v^{-1}_i}.\\
  (R_34)& & \sum_{k+k'=1-a_{ij}}(-1)^k\begin{bmatrix}
    1-a_{ij}\\k
  \end{bmatrix}_{v_i}(v^{i\cdot j}q_{ij})^{-k}E_i^kE_j E_i^{k'}=0,\quad {\rm if}\ i\not =j,\\
  & & \sum_{k+k'=1-a_{ij}}(-1)^k\begin{bmatrix}
    1-a_{ij}\\k
  \end{bmatrix}_{v_i}(v^{i\cdot j}q_{ij})^{-k}F_i^{k'}F_j F_i^{k}=0,\quad {\rm if}\ i\not =j.
\end{eqnarray*}

Let $U_{\mathbf{q}}$ be the multi-parameter quantum algebra in \cite{HPR} associated to $(I,\cdot)$.
\begin{prop}
 Under the assumption \eqref{qii}, the map $\Psi: U_{\mathbf{q}}\rightarrow \mbf U_{v,\mathbf{q}}/\langle J_i-1, J'_i-1\rangle$ 
 sending $e_i\mapsto E_i$, $f_i\mapsto F_i$, $\omega_{i}\mapsto -v_iK_i$ and $\omega'_i\mapsto -v_iK'_i$,  
 for all $i\in I$ is a $\mathbb{Q}$-algebra isomorphism.
\end{prop}
By Proposition \ref{proptwoparameter}, if we set $J_i=J'_i=1$ in Section \ref{sec3}, 
then the construction in Section \ref{sec3} provides a Drinfeld double construction of $U_{\mathbf{q}}$. 
This recovers the result in \cite{HPR}.

\subsection{The entire  Multi-parameter quantum superalgebras I}
  In the case of the multi-parameter quantum super algebra (see Example \ref{exmulsuper1}), 
  we further assume that $s_{ij}=s_{ji}$. Then, we set
\[\xi(i,j)=s_{ij}^{-1}\quad {\rm and}\quad \beta(i,j)=s_{ji}^{-1}p_{ji}v_i^{-a_{ij}}\ ,\forall i, j\in I.\]
Let $\mathcal{G}(\nu)=\prod_{i,j\in I}s_{ij}^{-\nu(i)\nu(j)/2}$, for all $\nu\in \mbb N[I]$ 
(i.e. $\nu=\sum_i\nu(i)i$), then \eqref{eqg} holds.
Under these settings, we obtain an entire multi-parameter quantum superalgebra $\mbf U_{\mathbf{s},\mathbf{p}}$,
generated by symbols $E_i, F_i, K_i^{\pm 1}, K_i'^{\pm 1}, J_i^{\pm 1}, J_i'^{\pm 1}$, 
$\forall i\in I$, and subjects to the following relations.
\allowdisplaybreaks
\begin{eqnarray*}
  (R_41) & &J_i^{\pm 1}, J_i'^{\pm 1}, K_i^{\pm 1}\ \text{and}\  K'^{\pm 1}_i\ \text{commute with each other},\\
 & &    K_i^{\pm 1}K_i^{\mp 1}=K'^{\pm 1}_iK'^{\mp 1}_i=J_i^{\pm 1}J_i^{\mp 1}=J_i'^{\pm 1}J_i'^{\mp 1}=1.\\
  (R_42)& &K_iE_j=p_{ij}^{-1} E_jK_i,\quad K'_iE_j=p_{ji}s_{ji}^{-1} E_jK'_i,\\
     & &K_iF_j=p_{ij} F_jK_i,\quad K'_iF_j=p_{ji}^{-1} s_{ji} F_jK'_i,\\
     & &J_iE_j=s_{ji}^{-1}E_jJ_i,\quad  J'_iE_j=E_jJ'_i,\\
     & &J_iF_j=s_{ji}F_jJ_i,\ \ J'_iF_j=F_jJ'_i.\\
  (R_43)& &E_iF_j-s_{ji}F_j E_i=\delta_{ij}\frac{K_iJ'_i-J_iK'_i}{v_i-v^{-1}_i}.\\
  (R_44)& & \sum_{k+k'=1-a_{ij}}(-1)^k\begin{bmatrix}
    1-a_{ij}\\k
  \end{bmatrix}_{v_i}s_{ji}^kp_{ji}^{-k}v_i^{ka_{ij}} E_i^kE_j E_i^{k'}=0,\quad {\rm if}\ i\not =j,\\
  & & \sum_{k+k'=1-a_{ij}}(-1)^k\begin{bmatrix}
    1-a_{ij}\\k
  \end{bmatrix}_{v_i}s_{ji}^kp_{ji}^{-k}v_i^{ka_{ij}}F_i^{k'}F_j F_i^{k}=0,\quad {\rm if}\ i\not =j.
\end{eqnarray*}

By \eqref{conapply} and the assumption $s_{ij}=s_{ji}$, we have
\begin{equation}\label{pij}
p_{ij}=p_{ji}s_{ji}^{-2},\quad \forall i,j \in I.
\end{equation}

Let $U_{\mathbf{s},\mathbf{p}}$ be the multi-parameter quantum superalgebra defined in \cite[Section 2]{KKO13} 
associated to $(I,\cdot)$ which is a bar-consistent super Cartan datum.
Let $\mcal J$ be the ideal of $\mbf U_{\mathbf{s},\mathbf{p}}$ generated by $\{K_iK_i'J_iJ_i'-1,\ \forall i\in I\}$,
and $\mbf U_{\mathbf{s},\mathbf{p}}^{sq}$ the subalgebra of $\mbf U_{\mathbf{s},\mathbf{p}}/\mcal J$ 
generated by $E_i, F_i, (K_i'J_i)^{\pm} $ and $(K_iJ_i')^{\pm}, \forall i \in I$.
\begin{prop}
 Under the assumptions \eqref{pi} and $s_{ij}=s_{ji}$, the map 
 $\Gamma: U_{\mathbf{s},\mathbf{p}}\rightarrow \mbf U_{\mathbf{s},\mathbf{p}}^{sq}$ 
 sending $e_i\mapsto E_i$, $f_i\mapsto F_i$, $K_i\mapsto -h_i K'_iJ_i$ for all $i\in I$ is a $\mbb Q$-algebra isomorphism.
\end{prop}

\subsection{The entire  Multi-parameter quantum superalgebras II}
  In the case of the multi-parameter quantum superalgebra (see Example \ref{exmulsuper2}), 
  we further assume that $\tilde{s}_{ij}=\tilde{s}_{ji}$. Then, we set
\[\xi(i,j)=\tilde{s}_{ij}^{-1}\quad {\rm and}\quad \beta(i,j)=\tilde{s}_{ij}\tilde{p}_i^{a_{ij}/2}\ ,\forall i, j\in I.\]

Let $\mathcal{G}(\nu)=\prod_{i,j\in I}\tilde{s}_{ij}^{-\nu(i)\nu(j)/2}$, 
for all $\nu\in \mbb N[I]$ (i.e., $\nu=\sum_i\nu(i)i$), then \eqref{eqg} holds.

Under these settings, we obtain an entire multi-parameter quantum superalgebra 
$\mbf U_{\tilde{\mathbf{s}},\tilde{\mathbf{p}}}$, generated by symbols 
$E_i, F_i, K_i^{\pm 1}, K_i'^{\pm 1}, J_i^{\pm 1}, J_i'^{\pm 1}$, $\forall i\in I$, and subjects to the following relations.
\allowdisplaybreaks
\begin{eqnarray*}
  (R_51) & &J_i^{\pm 1}, J_i'^{\pm 1}, K_i^{\pm 1}\ \text{and}\  K'^{\pm 1}_i\ \text{commute with each other},\\
 & &    K_i^{\pm 1}K_i^{\mp 1}=K'^{\pm 1}_iK'^{\mp 1}_i=J_i^{\pm 1}J_i^{\mp 1}=J_i'^{\pm 1}J_i'^{\mp 1}=1.\\
  (R_52)& &K_iE_j=E_jK_i,\quad K'_iE_j=\tilde{s}_{ij}\tilde{p}_i^{a_{ij}} E_jK'_i,\\
     & &K_iF_j=F_jK_i,\quad K'_iF_j=\tilde{s}_{ij}^{-1}\tilde{p}_i^{-a_{ij}} F_jK'_i,\\
     & &J_iE_j=\tilde{s}_{ji}^{-1}E_jJ_i,\quad  J'_iE_j=E_jJ'_i,\\
     & &J_iF_j=\tilde{s}_{ji}F_jJ_i,\ \ J'_iF_j=F_jJ'_i.\\
  (R_53)& &E_iF_j-\tilde{s}_{ji}F_j E_i=\delta_{ij}\frac{K_iJ'_i-J_iK'_i}{v_i-v^{-1}_i}.\\
  (R_54)& & \sum_{k+k'=1-a_{ij}}(-1)^k\begin{bmatrix}
    1-a_{ij}\\k
  \end{bmatrix}_{v_i}(\tilde{s}_{ij}\tilde{p}_{i}^{a_{ij}/2})^{-k} E_i^kE_j E_i^{k'}=0,\quad {\rm if}\ i\not =j,\\
  & & \sum_{k+k'=1-a_{ij}}(-1)^k\begin{bmatrix}
    1-a_{ij}\\k
  \end{bmatrix}_{v_i}(\tilde{s}_{ij}\tilde{p}_{i}^{a_{ij}/2})^{-k}F_i^{k'}F_j F_i^{k}=0,\quad {\rm if}\ i\not =j.
\end{eqnarray*}

Let $U_{\tilde{\mathbf{s}},\tilde{\mathbf{p}}}$ be the multi-parameter quantum superalgebra 
defined in \cite[Section 3]{KKO13} associated to $(I,\cdot)$ which is a bar-consistent super Cartan datum.
Let $\mcal L$ be the ideal of $\mbf U_{\tilde{\mathbf{s}},\tilde{\mathbf{p}}}$ generated by $\{K_iJ'_i+v_i^{-1},\ \forall i\in I\}$, and
$\mbf U_{\tilde{\mathbf{s}},\tilde{\mathbf{p}}}^{sq}$ the subalgebra of 
$\mbf U_{\tilde{\mathbf{s}},\tilde{\mathbf{p}}}/\mcal L$ generated by $E_i, F_i$ and $(K_i'J_i)^{\pm}, \forall i\in I$.
\begin{prop}
 Under the assumptions \eqref{tildepi} and $\tilde{s}_{ij}=\tilde{s}_{ji}$, 
 the map $\sigma: U_{\tilde{\mathbf{s}},\tilde{\mathbf{p}}}\rightarrow \mbf U_{\tilde{\mathbf{s}},\tilde{\mathbf{p}}}^{sq}$ 
 sending $e_i\mapsto E_i$, $f_i\mapsto F_i$ and  $\tilde K_i\mapsto -v_iK_i'J_i$ for all $i\in I$ is a $\mbb Q$-algebra isomorphism.
\end{prop}


\begin{thebibliography}{99999}\frenchspacing

\bibitem[BGH]{BGH}
N.~Bergeron, Y.~Gao, and N.~Hu, \emph{Drinfel'd doubles and Lusztig's symmetries of two-parameter quantum groups}. Journal of Algebra, 2005, 301(1):378-405.

\bibitem[BW]{BW}
G. Benkart and S. Witherspoon, \emph{Two-parameter quantum groups and Drinfel'd doubles}. Algebras \& Representation Theory, 2004, 7(3):261-286.

\bibitem[C17]{C}
S Clark, \emph{Odd knot invariants from quantum covering groups}. Algebraic \& Geometric Topology, 2017, 17(5):2961-3005.

\bibitem[CFLW]{CFLW}
S.~Clark, Z.~Fan, Y.~Li, and W.~Wang, \emph{Quantum supergroups III. Twistors}. Communications in Mathematical Physics, 2014, 332(1):415-436.

\bibitem[CHW]{CHW1}
S.~Clark, D.~Hill, and W.~Wang, \emph{Quantum supergroups I. foundations}. Transformation Groups, 2013, 18(4):1019-1053.

\bibitem[D87]{Dri87}
V.G. Drinfeld, \emph{Quantum groups}. In Proc. Int. Cong. Math. (Berkeley, 1986), pages 798-820. Amer. Math. Soc., Providence, RI, 1987. MR 89f:17017.


\bibitem[FL]{FL1}
Z.~Fan and Y.~Li, \emph{Two-parameter quantum algebras, canonical bases and categorifications}. International Mathematics Research Notices, 2015, 2015(16):7016-7062.

\bibitem[FLL]{FLL}
Z.~Fan, Y.~Li, and Z.~Lin, \emph{An identification of Lusztig's modified forms of quantum algebras and their analogues}. 2013, arXiv:1312.0838.


\bibitem[GHZ]{GHZ}
Y.~Gao, N.~Hu, and H.~Zhang, \emph{Two-parameter quantum affine algebra of type $G_2^{(1)}$, Drinfeld realization and vertex representation}. Journal of Mathematical Physics, 2015, 56(1):619-634.

\bibitem[HPR]{HPR}
N. Hu, Y. Pei, and M. Rosso, \emph{Multi-parameter quantum groups and quantum shuffles. I}. Quantum
affine algebras, extended affine Lie algebras, and their applications, 2010, 506:145-171.

\bibitem[HRZ]{HRZ}
N.~Hu, M.~Rosso, and H.~Zhang, \emph{Two-parameter quantum affine algebra $U_{r,s}(\widehat{\mathfrak{sl}_n})$, Drinfel'd realization and quantum affine Lyndon basis}.
Communications in Mathematical Physics, 2008, 278(2):453-486.


\bibitem[HZ]{HZ}
N.~Hu and H.~Zhang, \emph{Two-parameter quantum affine algebra of type $C_n^{(1)}$, Drinfeld realization and vertex representation}. Journal of Algebra, 2016, 459:43-75.

\bibitem[KKO]{KKO13}
S.-J. Kang, M. Kashiwara, and S.-J. Oh, \emph{Supercategorification of quantum Kac-Moody algebras II}. Advances in Mathematics, 2014, 265(1):169-240.

\bibitem[L10]{Lusztigbook}
G.~Lusztig, \emph{Introduction to quantum groups}. Modern Birkh$\ddot{a}$user Classics, Birkh$\ddot{a}$user$/$Springer, New York, 2010.


\bibitem[X97]{Xiao1}
J.~Xiao, \emph{Drinfeld double and Ringel-Green Theory of Hall algebras}. Journal of Algebra, 1997, 190(1):100-144.
\end{thebibliography}
\end{document}